\documentclass[12pt, a4paper]{article}
\usepackage[utf8]{inputenc}
\usepackage[T1]{fontenc}
\usepackage{amssymb,amsmath,amsrefs,amsthm} % Mathematical typesetting
\usepackage{graphicx}                   % Figures
\usepackage[a4paper,left=2.5cm,right=2.5cm,top=2.5cm,bottom=2.5cm]{geometry}                   % Margins
\usepackage{hyperref}                   % Hyperlinks
\usepackage{booktabs}                   % Professional tables
\usepackage{caption}
\usepackage{subcaption}
\usepackage{float}
\usepackage{tikz}                       % Drawing diagrams
\usepackage{siunitx}                    % Number alignment
\usepackage{xcolor}
\providecommand{\keywords}[1]{\textbf{Keywords:} #1}
\usepackage{tablefootnote}
\usepackage{enumitem}
\usepackage{booktabs}
\usepackage{multirow}
\usepackage{mathrsfs}
\usepackage{changes}
\usepackage[left]{lineno}
\hypersetup{colorlinks=true,
            linkcolor=blue,
            anchorcolor=blue,
            citecolor=blue}
\usetikzlibrary{arrows.meta, positioning, shapes, shadows, calc, decorations.pathreplacing}

\usepackage{chngcntr}
\counterwithin{figure}{section}
\newtheorem{theorem}{Theorem}[section]

\newtheorem{lemma}[theorem]{Lemma}

\newtheorem{remark}[theorem]{Remark}

\title{Concave Kite Central Configurations in the Planar Four-Body Problem with Three Equal Masses}
\author{Yangshanshan Liu\\
Chern Institute of Mathematics, Nankai University \\
Tianjin, 300071, China\\ % lyss6133@nankai.edu.cn\\
Nishan Poudel, Rupak Raut, Zhifu Xie\\
School of Mathematics and Natural Science\\
The University of Southern Mississippi\\
Hattiesburg, MS 39406, USA %, Zhifu.Xie@usm.edu\\ %<Nishan.Poudel@usm.edu>%<Rupak.Raut@usm.edu>}
}

\date{\today}
\graphicspath{{3+1figures/}}
\begin{document}
%\linenumbers
\modulolinenumbers[5] 
\maketitle

\begin{abstract}
We present a complete classification of concave kite central configurations in the planar 4-body problem with three equal masses. 
There are two different types of central configurations when the fourth mass lies inside or outside the triangle formed by the other three.
Using a rigorous computer-assisted analytical method and a fixed coordinate system, we show that the central configurations in each case form a one-parameter family and obtain a complete classification of these configurations. 
In addition, we rigorously show the existence and types of the bifurcation points in the reduced space. We also provide two numerical global bifurcation pictures in the entire planar 4-body configuration space as the mass ratio varies from $0$ to $+\infty$, including symmetric and asymmetric concave central configurations with three equal masses. 
\end{abstract}
\tableofcontents

\section{Introduction}

The Newtonian $n$-body problem seeks to describe the dynamics of $n$ point positive masses $m_i$ with positions $q_i=(q_{i1},\cdots,q_{id})^T \in \mathbb{R}^{d}\,(i=1,\cdots,n;\, d=2,3)$ moving under their mutual gravitational attraction
\begin{equation}\label{newton}
	m_i\ddot q_{i} =\frac{\partial U}{\partial q_{i}}=\sum \limits_{j\neq i\atop j=1}^n \frac{m_im_j(q_j-q_i)}{\vert q_j-q_i \vert^3},\quad i=1,\cdots,n,
\end{equation}
where 
\begin{equation*}
	U=\sum\limits_{i<j}\frac{m_im_j}{\vert q_j-q_i \vert}=\sum\limits_{i<j}\frac{m_im_j}{r_{ij}}
\end{equation*}
is the Newtonian potential and $r_{ij}=|q_i-q_j|$ is the mutual distance between any two bodies. 

The simplest planar solutions for \eqref{newton} are called \textit{homographic solutions}, in which each body describes a Keplerian orbit with respect to the center of mass $c=\frac{1}{M}\sum_{i=1}^nm_iq_i$ with $M=\sum _{i=1}^nm_i$ the total mass. The initial configuration of the bodies is invariant up to translation, scaling, and rotation. This special configuration $q = (q_1, \dots, q_n) \in \mathbb{R}^{d\times n}$ is called a \textit{central configuration}, satisfying that the acceleration vector of each body is proportional to its position vector with respect to $c$. 
Mathematically, this condition is satisfied if there exists a multiplier $\lambda>0$ such that
\begin{equation} \label{eq:cc_def}
    -\lambda m_{i}(q_i - c) = \sum_{j \neq i} \frac{m_i m_j (q_j - q_i)}{|q_i - q_j|^3} \quad i=1,\dots,n.
\end{equation}
Central configurations are of fundamental importance in celestial mechanics. 
Besides constructing homographic solutions, central configurations are also deeply related to the topology of the integral manifolds, and the dynamical behavior of the $n$-body system near collision. 

Solving the central configuration equations \eqref{eq:cc_def} is a challenging task, dating back to Euler and Lagrange, who found all five central configurations of the 3-body problem. 
A natural question for central configurations was considered by  Chazy \cite{chazy1918}, Winter \cite{wintner1941} in the early 1900s, and revived by Smale \cites{smale1970,smale1998}: \textit{The number of equivalence classes of central configurations of the $n$-body problem is finite for any given $n$ positive masses.} For the case $n=4$, Hampton and Moeckel \cite{hampton2006} gave a positive answer in 2006. 
For $n=5$, Hampton and Jensen \cite{hampton2011}, Albouy and Kaloshin \cite{albouy2012}, derived generic results for the spatial and the planar case, respectively. Chang and Chen \cite{chang2024} discussed the $n=6$ case, and there is still some distance to go before obtaining general finiteness. Much less is known for the case $n\geq7$. More evidence comes from the study of the special cases, such as the $n$-body collinear central configurations \cite{moulton1910}, and the equal mass cases for $n$ not very large \cites{albouy1996,lee2009,moczurad2019}. 

Significant progress has been made for the 4-body central configurations, especially the convex shapes. 
Given any four ordered masses, MacMillan and Bartky \cite{macmillan1932} showed the existence of a convex central configuration, and Xia \cite{xia2004} later gave a simpler proof. The uniqueness part, highlighted by Sim\'o and Yoccoz, remains open to a rigorous mathematical explanation. Special convex central configurations have been widely discussed \cites{cors2012, xie2012, santoprete2021a, santoprete2021b, roberts2025}, and the complete classification is well established \cite{corbera2019}, all of which provide positive evidence for this conjecture. 
Recently, Sun, Xie, and You \cite{sun2023} applied a computer-assisted approach to this conjecture and made some progress. 

For the concave case, Hampton \cite{hampton2002} gave the existence theorem in his thesis. Concave cases involving equal masses and symmetries have been discussed or mentioned in \cites{long2003,shi2010,bernat2009,martha2013, liu2026}, and the limit cases when several masses tend to zero have been discussed in \cites{hall1988,xia1991,casasayas1994,corbera2014,moeckel1997}. 
Since degenerate central configurations may arise in these cases, the bifurcation study is also considered in \cites{rusu2016, corbera2015, liu2026, meyer1988, simo1978, bernat2009, leandro2003}.

A \textit{kite} central configuration, or shortly, a kite, possesses a symmetric axis passing through two of the four masses. It may be convex or concave. 
Roberts \cite{roberts2025} showed the existence and the uniqueness of the convex kite, as well as the discussion on its linear stability and the bifurcation aspect.
The kite shape with two couples of equal masses has been fully discussed. A convex central configuration possessing two opposite equal masses must be a kite \cites{albouy2008,long2002}, or more precisely, a rhombus, according to the study by Bernat, Llibre, and P\'{e}rez-Chavela. Liu and Xie \cite{liu2026} found that the concave kite forms a one-parameter family, which provides an explicit picture of the full classification for any given non-negative mass ratio, incorporating the related results from \cites{martha2013,corbera2014}. In addition, the bifurcation of this concave kite type has been discussed in \cites{rusu2016,liu2026}.

In this paper, we focus on the concave kite central configurations in the 4-body problem when three of the four masses are equal. From \cite{bernat2009} we know that for a concave kite with three equal masses, the fourth mass should lie on the symmetry axis, and this leads to two distinct positions for $m_4$, corresponding to two different equation systems, as shown in Figures \ref{concaveouter} and \ref{concaveinner}. 
Unless otherwise specified, we set $m_1=m_2=m_3=1$ and $m_4=\mu$ during the discussion. 
We say that a 4-body concave kite central configuration with three equal masses is an \textbf{\textit{outer case}} if $m_4$ lies outside the triangle formed by the other three; it is an \textbf{\textit{inner case}} if $m_4$ lies inside the triangle. With this terminology, the previous results on this type can be clearly summarized.
Bernat et al. \cite{bernat2009} provided a numerical discussion of the classification, including both convex and concave cases. Shi and Xie \cite{shi2010} found the one-parameter property for the inner case. 
\begin{figure}[htbp]
        \begin{minipage}[t]{0.4\textwidth}
        \centering
        \includegraphics[width=\textwidth]{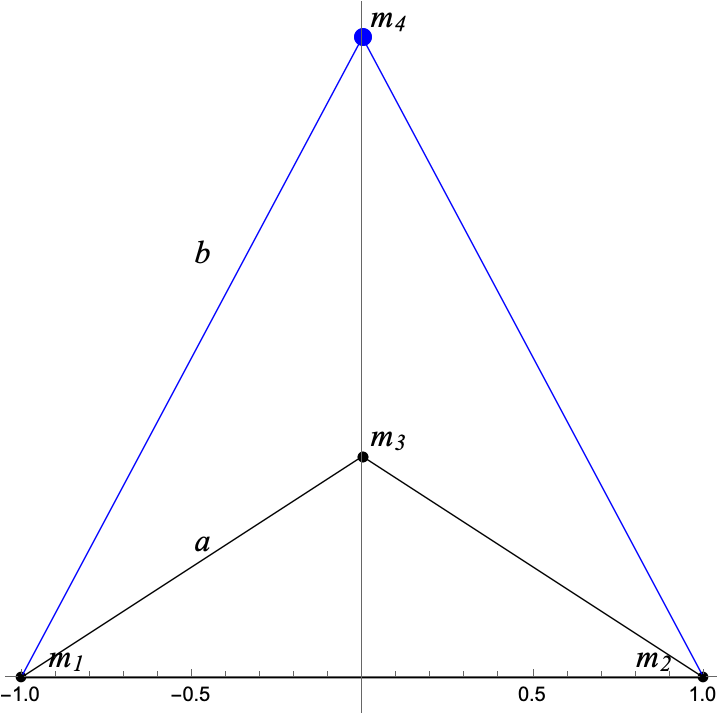}
        \caption{The outer case when $m_4$ lies outside the triangle formed by $m_1,m_2$ and $m_3$.}
        \label{concaveouter}
    \end{minipage}
      \hfill
    \begin{minipage}[t]{0.4\textwidth}
        \centering
       \includegraphics[width=\textwidth]{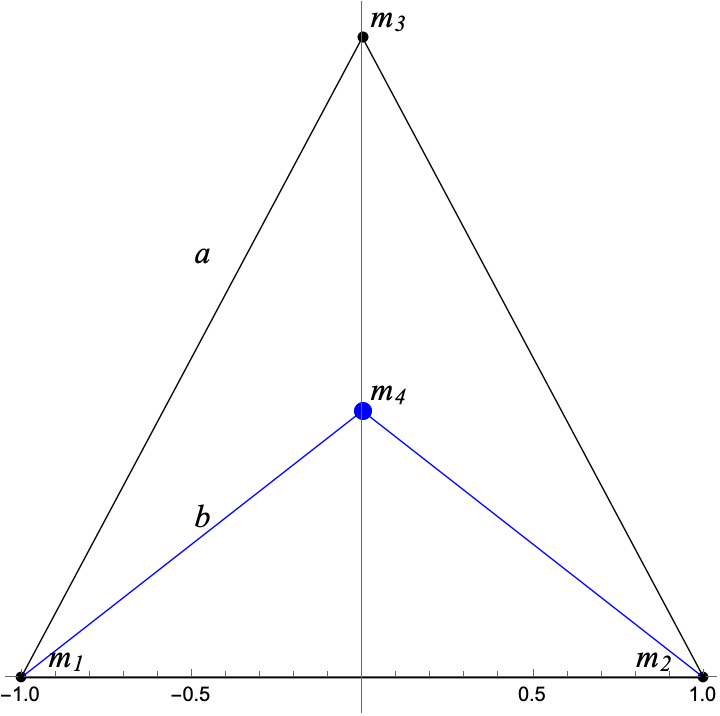}
        \caption{The inner case when $m_4$ lies inside the triangle formed by $m_1,m_2$ and $m_3$.}
        \label{concaveinner}
    \end{minipage}
\end{figure}
Compared with previous work, we use a rigorous, computer-assisted approach based on the interval method and the Krawczyk operator to show that the one-parameter property holds in both the outer and inner cases. This leads to a clear and complete classification for the concave kite central configurations with three equal masses for any given $\mu\geq0$. 

Our results are summarized in the following
\begin{theorem}\label{th1}
For the planar 4-body concave kite central configurations with three equal masses, the outer case admits a one-parameter family. In the inner case, besides the equilateral triangle family with three equal masses at the vertices and an arbitrary mass at the center, there exists a one-parameter isosceles triangle family with three equal masses at the vertices and an arbitrary mass on the symmetric axis.
\end{theorem}
Based on Theorem \ref{th1} and the results in \cites{bernat2009,shi2010}, we obtain the classification results, in which the mass ratio $\mu$ can be extended to $0$ and $+\infty$.
\begin{theorem}\label{th2}
Supposing that a concave kite central configuration satisfies $m_1=m_2=m_3=1$ and $m_4=\mu$. Then the following hold: 
\begin{enumerate}
	\item For the outer case, 
	\begin{enumerate}[label=1.\arabic*)]
		\item there are precisely two central configurations if $\mu\in (0,m_{O})$;
		\item there is only one central configuration if $\mu=m_{O}$;
		\item there is no such kite central configuration if $\mu>m_{O}$,
    \end{enumerate}
where $m_{O}\approx1.0026605475726100006858035$.
	\item For the inner case,
	\begin{enumerate}[label=2.\arabic*)]
		\item there are exactly two central configurations when $\mu\in (0,m_I)\cup (m_I,+\infty)$ corresponding to the equilateral and the isosceles families respectively; 
		\item there is only one central configuration if $\mu=m_I$, which is the intersection of the above two families, possessing an equilateral triangle shape;
	\end{enumerate}
	where $m_I=\frac{64 \sqrt{3}+81}{249}$.
	\item In addition, if $\mu=0$ is admissible, 
	  \begin{enumerate}[label=3.\arabic*)]
	    \item there are exactly two outer cases: one with $m_3$ positioned between $m_1$ and $m_2$, and the other one with these three masses forming an equilateral triangle, as shown in blue in Figures \ref{a=1} and \ref{a=2};
	    \item there are exactly two inner cases, where $m_1,m_2,m_3$ form an equilateral triangle as shown in red in Figure \ref{a=2} .
      \end{enumerate}
\end{enumerate}
\end{theorem}
\begin{remark}
We point out here that the result concerning the inner case in Theorem \ref{th1} has already been discussed in \cite[Theorem 1.4]{shi2010}, and the difference lies in the choice of variables.
	The result in Theorem \ref{th2} can be reorganized in Table \ref{numberofcc}.
\begin{table}[htbp]
\centering
\begin{tabular}{c|ccccccc}
\toprule
$\mu$ & $0$ & $(0,m_I)$& $m_I$ &$(m_I,m_O)$ & $m_O$& $(m_O,+\infty)$&$+\infty$ \\
\midrule
Outer & 2& 2& 2& 2&\textcolor{blue}{1} &$-$ &$ -$\\
Inner & 2& 2& \textcolor{blue}{1}& 2&2 &2 & 2\\
\midrule
Total &4 & 4& 3& 4& 3&2 &2 \\
\bottomrule
\end{tabular}
\caption{With $m_1=m_2=m_3=1$, $m_4=\mu$ and $q_1=(-1,0),q_2=(1,0)$, we count the number of the concave kite central configurations with three equal masses.}
\label{numberofcc}
\end{table}
\end{remark}

%The threshold value $m_O$ was found in \cite{bernat2009}, and $m_I$ by Palmore \cite{palmore1975b}.

%%% bifurcation introduction
We further explore the bifurcations in the reduced symmetric subspace for the outer and the inner cases. To make it clear, we use the notations in Section \ref{secbif}. 
Using a computer-assisted approach, we rigorously show that in the outer case there is a unique fold bifurcation point $\mathcal{F}$ when $\mu=m_O$, and it is the one already found in \cite{bernat2009} (with notation $m^\ast$). 
For the inner case, there is a unique transcritical bifurcation point $\mathcal{T}$ when $\mu=m_I$, which is the famous one found by Palmore \cite{palmore1975b}. 

Based on the foregoing analysis and our numerical computations, we present two global bifurcation images covering the entire 4‑body planar configuration space. These correspond to the outer and inner cases, respectively, and account for both symmetric and asymmetric concave central configurations with three equal masses. The images are shown in Figures \ref{outerbif} and \ref{innerbif}, with the numerical values listed in Tables \ref{outervalues} and \ref{innervalues}. 
In the outer case, there exists a supercritical pitchfork bifurcation point $\mathcal{P}$ found in \cite{rusu2016} when $\mu=m_{P}\approx0.99184227439094091554349$ (or $m_{\ast\ast}$ in their notation), giving rise to the asymmetric concave central configurations when $\mu>m_{P}$. Furthermore, when $\mu\to +\infty$, numerical evidence shows that the limiting position for the asymmetric central configuration forms an equilateral triangle, with two of the equal masses colliding with each other, as shown in Figure \ref{outerbif}. 
In fact, the limiting case can be considered equivalently as the relative equilibria of the $(1+3)$-body problem when the only large mass $m_4$ is set to be 1, and the other three equal masses tend to zero. 
%A more detailed asymptotic process of this setting is shown in Figure \ref{minfinite}. 
%\deleted{Interestingly, 
%the limiting collision case indicated in Figure \ref{outerbif} is different from those non-collision ones discussed in \cite{corbera2015,hall1988,casasayas1994}. We will elaborate further in Section \ref{secbif}.}

The whole paper is organized as follows.
In Section \ref{3eq}, we introduce the central configuration equations for both outer and inner cases, and determine the proper ranges for the corresponding variables in Lemma \ref{rangeab}. In Sections \ref{outersec} and \ref{innersec}, we deal with both cases, respectively, including the proofs of Theorems \ref{th1} and \ref{th2}. At the end of each section, we study the bifurcation in the reduced symmetric subspace. In Section \ref{secbif}, we give a numerical analysis of the bifurcations in the whole planar 4-body configuration space.
We present the necessary theories used in this paper in Section \ref{sec2} for convenience, including interval arithmetic, the Krawczyk operator, and the bifurcation theory used in this context.

\section{Equations}\label{3eq}
We consider a planar configuration with an axis of symmetry along the $y$-axis. The masses are set to be $m_1 = m_2 = m_3 = 1$ and $m_4 = \mu$. The positions of the particles are parameterized as:
\begin{equation}\label{q1234}
    q_1 = (-1, 0), \quad q_2 = (1, 0), \quad q_3 = (0, \sqrt{a^2-1}), \quad q_4 = (0, \sqrt{b^2-1}),
\end{equation}
where $a,b\geq1$.
By substituting the above into \eqref{eq:cc_def} and simplifying, we derive four equations
\begin{subequations}
	\begin{align}
		%-\lambda +\frac{m}{b^3}+\frac{1}{a^3}+\frac{1}{4}&=0,\\ 
		\lambda -\frac{\mu}{b^3}-\frac{1}{a^3}-\frac{1}{4}&=0,\notag\\
		-\frac{\lambda  (\mu \sqrt{b^2-1}+\sqrt{a^2-1})}{\mu+3}+\frac{\mu \sqrt{b^2-1}}{b^3}+\frac{\sqrt{a^2-1}}{a^3}&=0,\label{origin2}\\
		%-\frac{\lambda  (m \sqrt{b^2-1}+\sqrt{a^2-1})}{m+3}+\frac{m \sqrt{b^2-1}}{b^3}+\frac{\sqrt{a^2-1}}{a^3}&=0,\\
		\lambda  \left(\sqrt{a^2-1}-\frac{\mu \sqrt{b^2-1}+\sqrt{a^2-1}}{\mu+3}\right)-\frac{\mu(\sqrt{a^2-1}-\sqrt{b^2-1})}{\vert \sqrt{a^2-1}-\sqrt{b^2-1}\vert^3}-\frac{2 \sqrt{a^2-1}}{a^3}&=0,\label{origin3}\\
		-\frac{\lambda  (\sqrt{a^2-1}-3 \sqrt{b^2-1})}{\mu+3}+\frac{\sqrt{a^2-1}-\sqrt{b^2-1}}{\vert \sqrt{a^2-1}-\sqrt{b^2-1}\vert^3}-\frac{2 \sqrt{b^2-1}}{b^3}&=0\label{origin4}.
	\end{align}
\end{subequations}
The relationship for the left-hand sides of \eqref{origin2},\eqref{origin3} and \eqref{origin4} is 
$$2\cdot\eqref{origin2}+\eqref{origin3}+\mu\cdot\eqref{origin4}=0.$$ 
Then we reduce the system to
\begin{equation*}
	\left\{
	\begin{aligned}
		\lambda -\frac{\mu}{b^3}-\frac{1}{a^3}-\frac{1}{4}&=0,\\
\frac{3 \sqrt{a^2-1}}{a^3}+\frac{\mu \sqrt{b^2-1}}{b^3}+\frac{\mu(\sqrt{a^2-1}-\sqrt{b^2-1})}{\vert \sqrt{a^2-1}-\sqrt{b^2-1}\vert^3}-\lambda  \sqrt{a^2-1}&=0,\\
\frac{\sqrt{a^2-1}}{a^3}+\frac{(\mu+2) \sqrt{b^2-1}}{b^3}-\frac{\sqrt{a^2-1}-\sqrt{b^2-1}}{\vert \sqrt{a^2-1}-\sqrt{b^2-1}\vert^3}-\lambda  \sqrt{b^2-1}&=0.
	\end{aligned}
	\right.
\end{equation*}
By substituting the expression of $\lambda$ from the first equation above into the last two, we get the following after some simplification
\begin{subequations}
	\begin{align}
		\lambda -\frac{\mu}{b^3}-\frac{1}{a^3}-\frac{1}{4}&=0,\label{cc1}\\
		\mu \left(\frac{\sqrt{a^2-1}-\sqrt{b^2-1}}{b^3}-\frac{\sqrt{a^2-1}-\sqrt{b^2-1}}{\vert \sqrt{a^2-1}-\sqrt{b^2-1}\vert^3}\right)-\frac{2 \sqrt{a^2-1}}{a^3}+\frac{\sqrt{a^2-1}}{4}&=0,\label{cc2}\\
\frac{\sqrt{a^2-1}-\sqrt{b^2-1}}{a^3}-\frac{\sqrt{a^2-1}-\sqrt{b^2-1}}{\vert \sqrt{a^2-1}-\sqrt{b^2-1}\vert^3}+\frac{2\sqrt{b^2-1}}{b^3}-\frac{\sqrt{b^2-1}}{4}&=0.\label{cc3}
	\end{align}
\end{subequations}
We consider the following expressions
\begin{equation}\label{xx1yy1}
	\left\{
	\begin{aligned}
		x=& \frac{\sqrt{a^2-1}(8-a^3)}{4a^3}\\%=\frac{8-a^3}{4a^3} \sqrt{a^2-1}\\
		x_1=&(\sqrt{a^2-1}-\sqrt{b^2-1})\frac{\vert \sqrt{a^2-1}-\sqrt{b^2-1}\vert^3-a^3}{a^3\vert \sqrt{a^2-1}-\sqrt{b^2-1}\vert^3}\\%=(\sqrt{a^2-1}-\sqrt{b^2-1})\frac{a^3+(\sqrt{a^2-1}-\sqrt{b^2-1})^3}{a^3(\sqrt{a^2-1}-\sqrt{b^2-1})^2}\\
		y=& \frac{\sqrt{b^2-1}(8-b^3)}{4b^3}\\%=\frac{8-b^3}{4b^3} \sqrt{b^2-1}\\
		y_1=&(\sqrt{a^2-1}-\sqrt{b^2-1})\frac{\vert \sqrt{a^2-1}-\sqrt{b^2-1}\vert ^3-b^3}{b^3\vert \sqrt{a^2-1}-\sqrt{b^2-1}\vert^3}.%=(\sqrt{a^2-1}-\sqrt{b^2-1})\frac{b^3+(\sqrt{a^2-1}-\sqrt{b^2-1})^3}{b^3\left(\sqrt{a^2-1}-\sqrt{b^2-1}\right)^2}.
	\end{aligned}
	\right.
\end{equation}
Then from \eqref{cc2}, \eqref{cc3}, and \eqref{xx1yy1} we obtain a mass expression $\mu=\mu(a,b)$ and a mass-free equation $g(a,b)=0$
\begin{subequations}
	\begin{align}
		\mu =& \mu(a,b)=\frac{x}{y_1}\label{mab}=\frac{(2-a)}{(\vert \sqrt{a^2-1}-\sqrt{b^2-1}\vert-b)(\sqrt{a^2-1}-\sqrt{b^2-1})}\cdot\check \mu>0,\\
		g = &g(a,b)= x_1+y\notag\\
		=&(\sqrt{a^2-1}-\sqrt{b^2-1})\cdot(\vert \sqrt{a^2-1}-\sqrt{b^2-1}\vert-a)\cdot \check g_1+(2-b)\cdot \check g_2=0\label{gab},
	\end{align} 
\end{subequations}
where we find the following are all positive
\begin{equation*}
\left\{
	\begin{aligned}
		\check \mu=&\frac{b^3\sqrt{a^2-1}\cdot\vert\sqrt{a^2-1}-\sqrt{b^2-1} \vert^3(a^2+2a+4)}{\vert\sqrt{a^2-1}-\sqrt{b^2-1} \vert^2+b\vert\sqrt{a^2-1}-\sqrt{b^2-1} \vert+b^2}\cdot\frac{1}{4a^3}>0,\\
		\check g_1=&\frac{\vert\sqrt{a^2-1}-\sqrt{b^2-1} \vert^2+a\vert\sqrt{a^2-1}-\sqrt{b^2-1} \vert+a^2}{a^3\vert\sqrt{a^2-1}-\sqrt{b^2-1} \vert^3}>0,\\
		\check g_2=&\frac{(b^2+2b+4)\sqrt{b^2-1}}{4b^3}>0.
	\end{aligned}
\right.
\end{equation*}
It is easy to see that the sign of $\sqrt{a^2-1}-\sqrt{b^2-1}$ determines two different equation systems for this concave shape.
When $a<b$, we get the outer case, and when $a>b$, we get the inner case. The corresponding expressions are denoted as follows
\begin{equation}\label{outereqs}
\left\{
\begin{aligned}
	\mu_{outer}(a,b) = &\mu(a,b)\\
	g_{outer} (a,b)= &g(a,b)
	\end{aligned}
\right.,\quad a<b
\end{equation}
\begin{equation}\label{innereqs}
\left\{
\begin{aligned}
	\mu_{inner}(a,b) = &\mu(a,b)\\
	g_{inner}(a,b)= &g(a,b)
	\end{aligned}
\right.,\quad a>b.
\end{equation}
For convenience, we denote by $h_{inner}(a,b)=\sqrt{a^2-1}-\sqrt{b^2-1}-b.$

\begin{lemma}\label{rangeab}
For a concave 4-body kite central configuration with three equal masses $m_1=m_2=m_3=1$ and $m_4=\mu$, and the positions in \eqref{q1234}, we have the following
\begin{enumerate}
	\item For the outer case, $(a,b)\in D_{outer}= D_{outer_1}\cup\{(2/\sqrt{3},2)\} \cup D_{outer_2}$ holds for $\mu>0$. More precisely, we have 
		\begin{enumerate}[label=1.\arabic*)]
			\item $D_{outer_1}=(1,2/\sqrt{3})\times (\sqrt{2},2)$;
			\item $D_{outer_2}=(2/\sqrt{3},2)\times (2,\sqrt{2}+\sqrt{6})$;
			\item when $(a,b)=(2/\sqrt{3},2)$, i.e., $m_1,m_2,m_4$ form an equilteral triangle with $m_3$ at the center. This holds for $\mu=1$. 
		\end{enumerate}
	\item For the inner case, $(a,b)\in D_{inner}= D_{inner_1} \cup \{(2,2/\sqrt{3})\}\cup D_{inner_2}$ holds for $\mu>0$. More precisely, we have
	\begin{enumerate}[label=2.\arabic*)]
		\item \label{2.1}$D_{inner_1}=(\sqrt{2},2)\times (\beta_3,2/\sqrt{3})$, where $\beta_3\approx1.004931366268962$ is the smaller root of the only two positive roots of
	 \begin{equation}\label{ginner2b}
 	g_{inner}(2,b)=\frac{3 \sqrt{b^2-1}}{8}+\frac{1}{\left(\sqrt{3}-\sqrt{b^2-1}\right)^2}-\frac{2 \sqrt{b^2-1}}{b^3}-\frac{\sqrt{3}}{8}=0
 \end{equation}
for $b\in (1,2)$;
		\item $D_{inner_2}=(2,\alpha_4)\times (2/\sqrt{3},\beta_4)$, where $(\alpha_4,\beta_4)$ is one of the intersections of $h_{inner}(a,b)=0$ and $g_{inner}(a,b)=0$ with $\alpha_4\approx2.839911839069655$ and $\beta_4\approx1.5171222826401447$;
		\item when $(a,b)=(2,2/\sqrt{3})$, i.e., $m_1,m_2,m_3$ form an equilateral triangle with $m_4$ at the center. This holds for arbitrary $\mu>0$. 
	\end{enumerate} 
	\item In addition, for the limiting case $\mu=0$, there are exactly four such concave central configurations(see Figure \ref{a=1} and \ref{a=2}). Two for the outer case: $(a,b)\in \{(1,\beta_1),(2,\beta_2)\}$, and two for the inner case:$(a,b)\in \{(2,\beta_3),(2,2/\sqrt{3})\}$, where 
\begin{enumerate}[label=3.\arabic*)]
\item $\beta_1\approx 1.5160134167700623$ is the only positive real root of 
\begin{equation}\label{gouter1b}
	g_{outer}(1,b)=\frac{5 \sqrt{b^2-1}}{4}+\frac{1}{1-b^2}-\frac{2 \sqrt{b^2-1}}{b^3}=0
\end{equation}
for $b\in (1,2)$ with $m_3$ at the middle of $m_1$ and $m_2$;
\item $\beta_2\approx 3.102903134537836$ is the only positive root of 
\begin{equation}\label{gouter2b}
	g_{outer}(2,b)=\frac{3 \sqrt{b^2-1}}{8}-\frac{1}{\left(\sqrt{3}-\sqrt{b^2-1}\right)^2}-\frac{2 \sqrt{b^2-1}}{b^3}-\frac{\sqrt{3}}{8}=0
\end{equation}
for $b\in (2,\sqrt{2}+\sqrt{6})$;
\item $2/\sqrt{3}$ is the larger root of the only two positive roots of $g_{inner}(2,b)=0$.
\end{enumerate}
\end{enumerate}
\end{lemma}

\begin{figure}[htbp]
        \begin{minipage}[t]{0.3\textwidth}
        \centering
        \includegraphics[width=\linewidth]{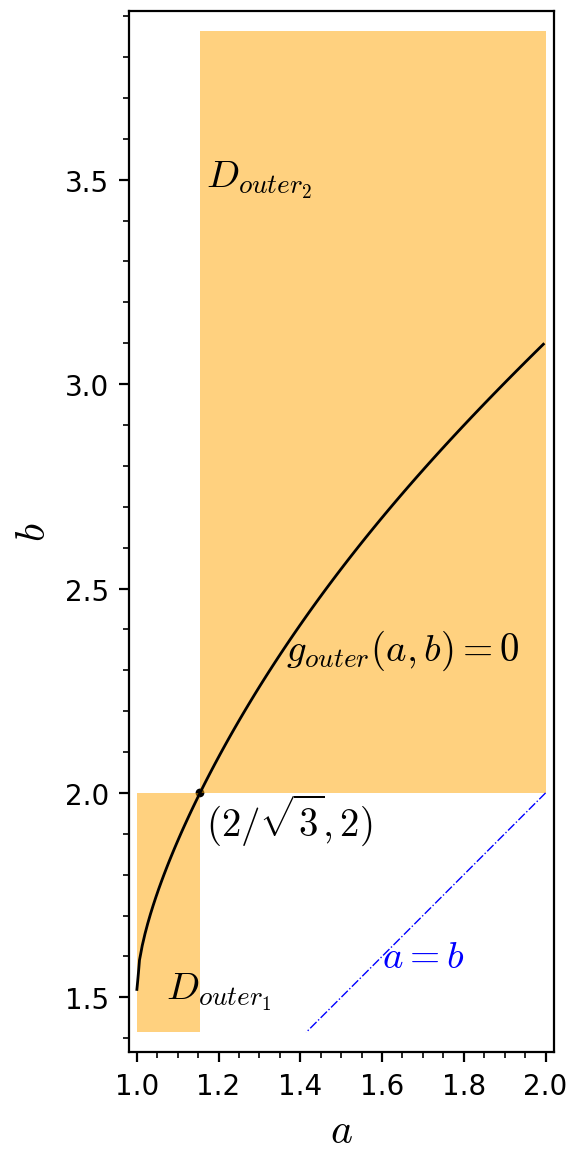}
\caption{The two orange regions denote the $D_{outer_1}$ and $D_{outer_2}$ respectively on the $aOb$-plane. The black curve denotes $g_{outer}(a,b)=0$.}
\label{outer_rec}
    \end{minipage}
      \hfill
    \begin{minipage}[t]{0.65\textwidth}
        \centering
        \includegraphics[width=\textwidth]{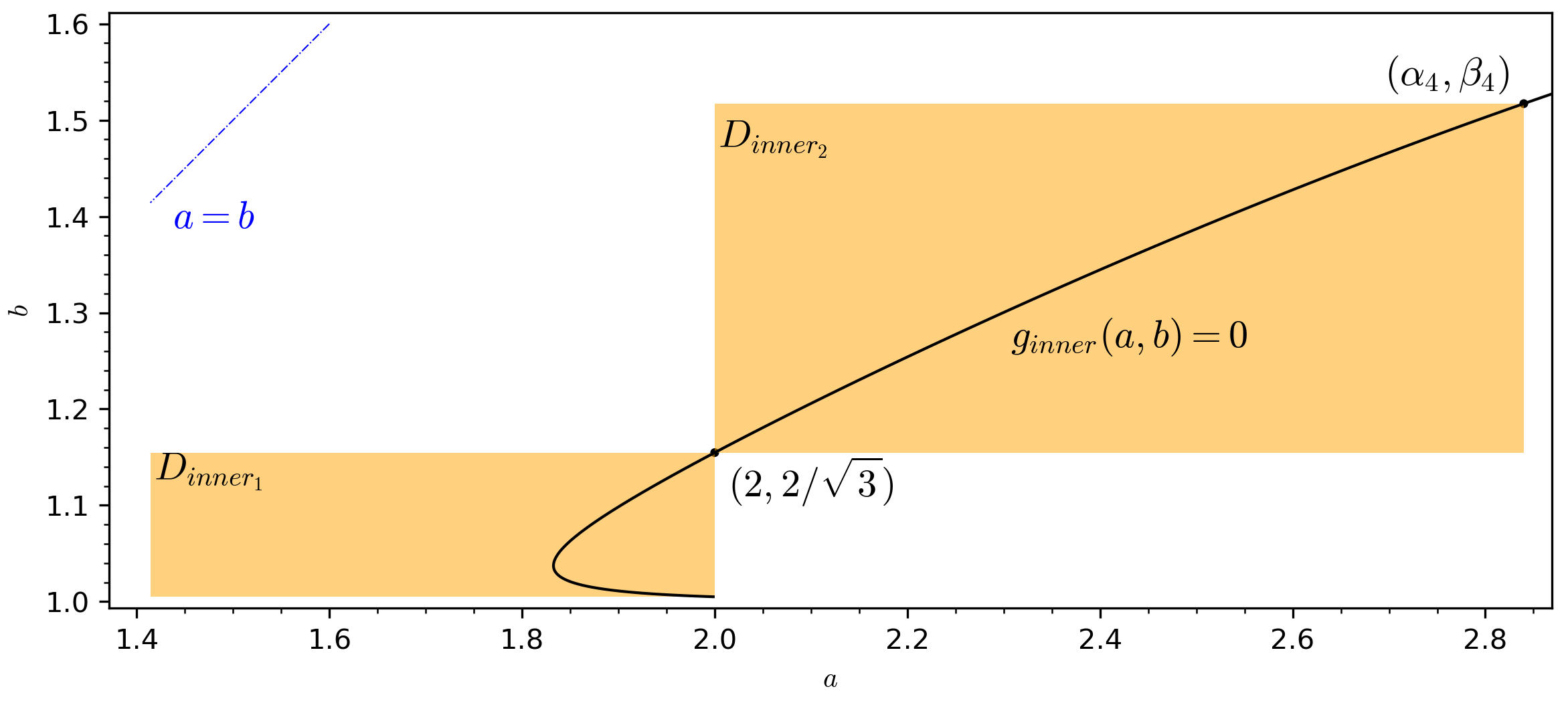}
        \caption{The two orange regions denote the $D_{inner_1}$ and $D_{inner_2}$ respectively on the $aOb$-plane. The black curve denotes $g_{inner}(a,b)=0$.}
        \label{inner_rec}
    \end{minipage}
\end{figure}

\begin{figure}[htbp]
\begin{minipage}[t]{0.47\textwidth}
        \includegraphics[width=\linewidth]{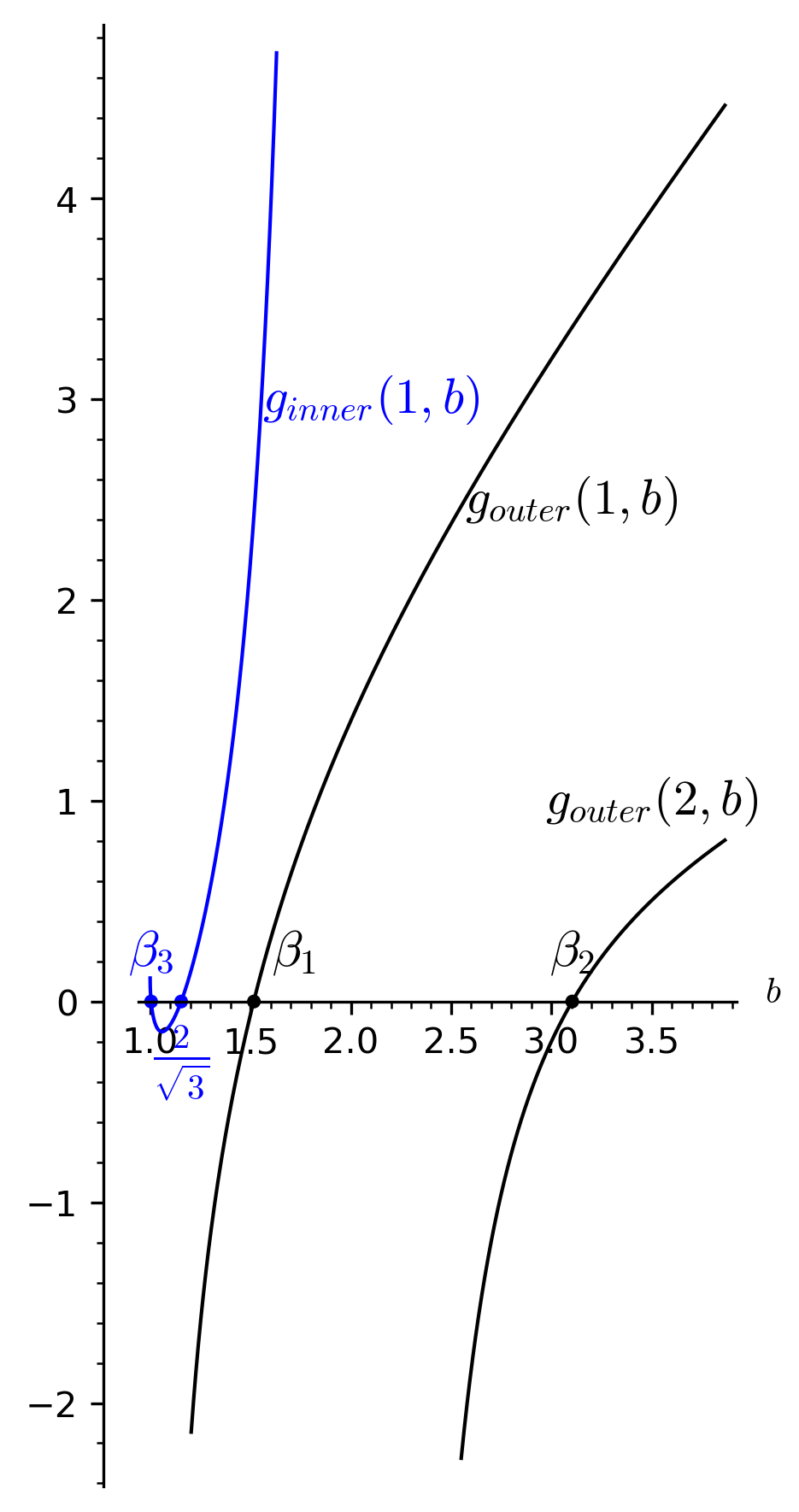}
\caption{The zeros of $g_{inner}(1,b)=0$, $g_{outer}(1,b)=0$ and $g_{outer}(2,b)=0$.}
\label{beta134}
\end{minipage}
\hfill
\begin{minipage}[t]{0.47\textwidth}
        \centering
        \includegraphics[width=\textwidth]{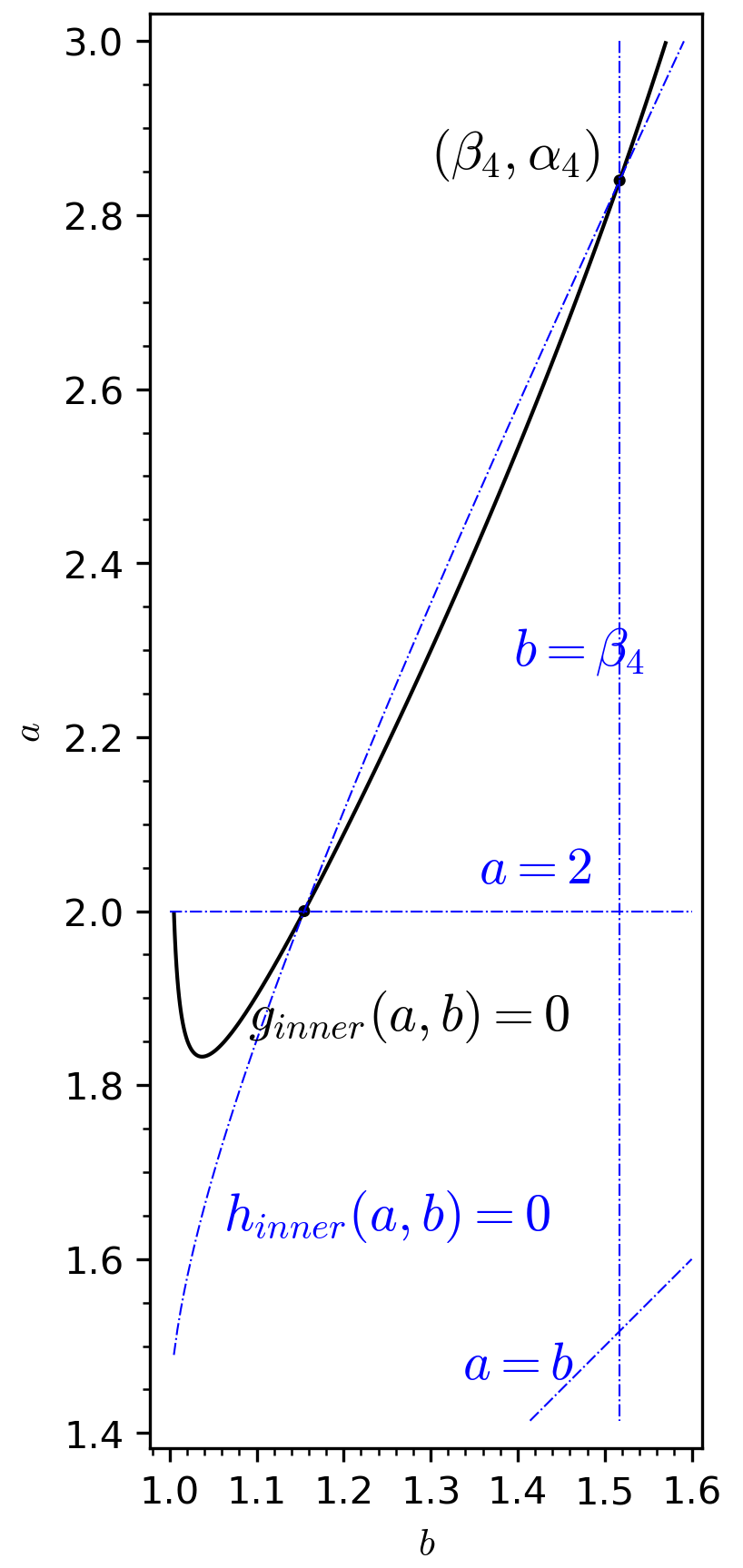}
        \caption{The two intersections $(2,2/\sqrt{3})$ and $(\alpha_4,\beta_4)$ of $g_{inner}(a,b)=0$ and $h_{inner}(a,b)=0$ on the $bOa$ plane.}
        \label{innergh}
\end{minipage}
\end{figure}

\begin{proof}
In the outer case, from the Perpendicular Bisector Theorem \cite{moeckel1990}, we find that it is the same as Corollary 1 in \cite{liu2026}, i.e., $a\in (1,2)$ and $b\in (\sqrt{2},\sqrt{2}+\sqrt{6})$. 
	Together with $a<b$, we have $(a,b)\in D_{outer}$. 

Although the inner case has been thoroughly discussed in \cite{shi2010} using different coordinates and variables, we still choose to present some details here to maintain consistency with the outer case. 

Firstly, by exchanging the positions of $a$ and $b$ in the outer case, we obtain $(a,b)\in \left((\sqrt{2},2)\times (1,2/\sqrt{3})\right)\cup \{(2,2/\sqrt{3})\}\cup \left((2,\sqrt{2}+\sqrt{6})\times (2/\sqrt{3},2)\right)$. One can see the orange regions in Figure \ref{outer_rec}.

Secondly, noticing that 
\begin{equation}\label{ginnerabtoa}
	\begin{aligned}
	&\partial g_{inner}(a,b)/\partial a\\
		=&-\frac{2 a}{\sqrt{a^2-1} \left(\sqrt{a^2-1}-\sqrt{b^2-1}\right)^3}-\frac{1}{a^2 \sqrt{a^2-1}}+\frac{3 \left(\sqrt{a^2-1}-\sqrt{b^2-1}\right)}{a^4}\\
		=&\frac{\sqrt{a^2-1}-\sqrt{b^2-1}}{a^4}-\frac{\sqrt{a^2-1}}{a^2 \sqrt{a^2-1}^2}\\
		&+\frac{2 \left(\sqrt{a^2-1}-\sqrt{b^2-1}\right)}{a^4}-\frac{2 a}{\sqrt{a^2-1} \left(\sqrt{a^2-1}-\sqrt{b^2-1}\right)^3}\\
		<&\frac{\sqrt{a^2-1}-\sqrt{b^2-1}-\sqrt{a^2-1}}{a^2(a^2-1)}+2\frac{\sqrt{a^2-1}-\sqrt{b^2-1}-a}{\sqrt{a^2-1} \left(\sqrt{a^2-1}-\sqrt{b^2-1}\right)^3}<0
	\end{aligned}
\end{equation}
holds for $1<b<a$, then we have $g_{inner}(a,b)>g_{inner}(2,b)$ with $a<2$.
This implies that the smallest positive zero of $g_{inner}(2,b)=0$ is the smallest positive zero of $g_{inner}(a,b)=0$ (if any) for any $a\in (\sqrt{2},2)$. Then by using the interval arithmetic method and the Krawczyk operator introduced in Section \ref{koperator}, we find the two positive roots for \eqref{ginner2b} easily, and the smaller one is $\beta_3$. When $b\leq\beta_3$, we have $g_{inner}(a,b)>0$ for any $a\in(\sqrt{2},2)$. Hence we have $(a,b)\in D_{inner_1}$. 

Thirdly, for 2.2), by using the interval arithmetic method and the Krawczyk operator, we show that $g_{inner}(a,b)=0$ and  $h_{inner}(a,b)=0$ have exactly two solutions when $(a,b)\in (\sqrt{2},\sqrt{2}+\sqrt{6})\times (\beta_3,2)$, as shown in Figure \ref{innergh}. We find $g_{inner}(3,b)=0$ has one zero in $b\in I_h=[1.570610354361449,1.570610354361450]$. By direct interval computation we have $h_{inner}(3,I_h)=[0.0466946108866150, 0.0466946108866194]\subset (0,+\infty)$. This implies that $h_{inner}(a,b)>0$ holds for the points $(a,b)$ on $g_{inner}(a,b)=0$ with $b>\beta_4$ and $a>2$. 
This contradicts $\mu_{inner}>0$ according to the expression \eqref{mab}. Hence we have $D_{inner_2}=(2,\alpha_4)\times (2/\sqrt{3},\beta_4)$. One can see the orange regions in Figure \ref{inner_rec}.

%For 2.1) and 2.2), by using the interval arithmetic method and the Krawczyk operator introduced in Section \ref{koperator}, we can find the two positive roots for \eqref{ginner2b}, and the two intersections of $g_{inner}(a,b)=0$ and $h_{inner}(a,b)=0$ without much difficulty, as shown in Figure \ref{beta134} and \ref{innergh}. 

%When $b>\beta_4$ we have $a>2$. Then for any $(a,b)$ satisfying $g_{inner}(a,b)=0$, we have $h_{inner}(a,b)>0$. This contradicts $m_{inner}>0$ according to the expression \eqref{mab}. Hence $b<\beta_4$.
Finally, for 2.3), from the expressions of \eqref{cc2} and \eqref{xx1yy1}, namely, $\mu\cdot y_1=x$, we have $x=0\Leftrightarrow a=2$ and $y_1=0\Leftrightarrow b=2/\sqrt{3}$. This holds for any $\mu>0$.

For the case $\mu=0$, a similar computation is done by substituting $\mu=0$ into \eqref{outereqs} and \eqref{innereqs}. By solving both systems of equations using the interval arithmetic method and the Krawczyk operator, we obtain the corresponding results. This also completes the third part of the proof for Theorem \ref{th2}.
\end{proof}
\begin{remark}\label{alphabeta4}
	When $g_{inner}(a,b)=0$ and $h_{inner}(a,b)=0$ intersect at $(2,2/\sqrt{3})$ (i.e., point $(a_I,b_I)$ in Figure \ref{innercase}) and $(\alpha_4,\beta_4)$, we find $\sqrt{a^2-1}-\sqrt{b^2-1}=b$, i.e., $r_{34}=r_{14}=r_{24}$, and the denominator of $\mu$ in \eqref{mab} vanishes, namely, $\mu\to +\infty$. This corresponds to the two concave central configurations of the $(1+3)$-body problem when the large mass $m_4$ is set to be 1, and the three small masses are equal. The limit positions of the three infinitesimal masses are on the circle centered at the only large mass.
\end{remark}
\begin{figure}[t]%[htbp]
        \begin{minipage}[t]{0.4\textwidth}
        \centering
        \includegraphics[width=\textwidth]{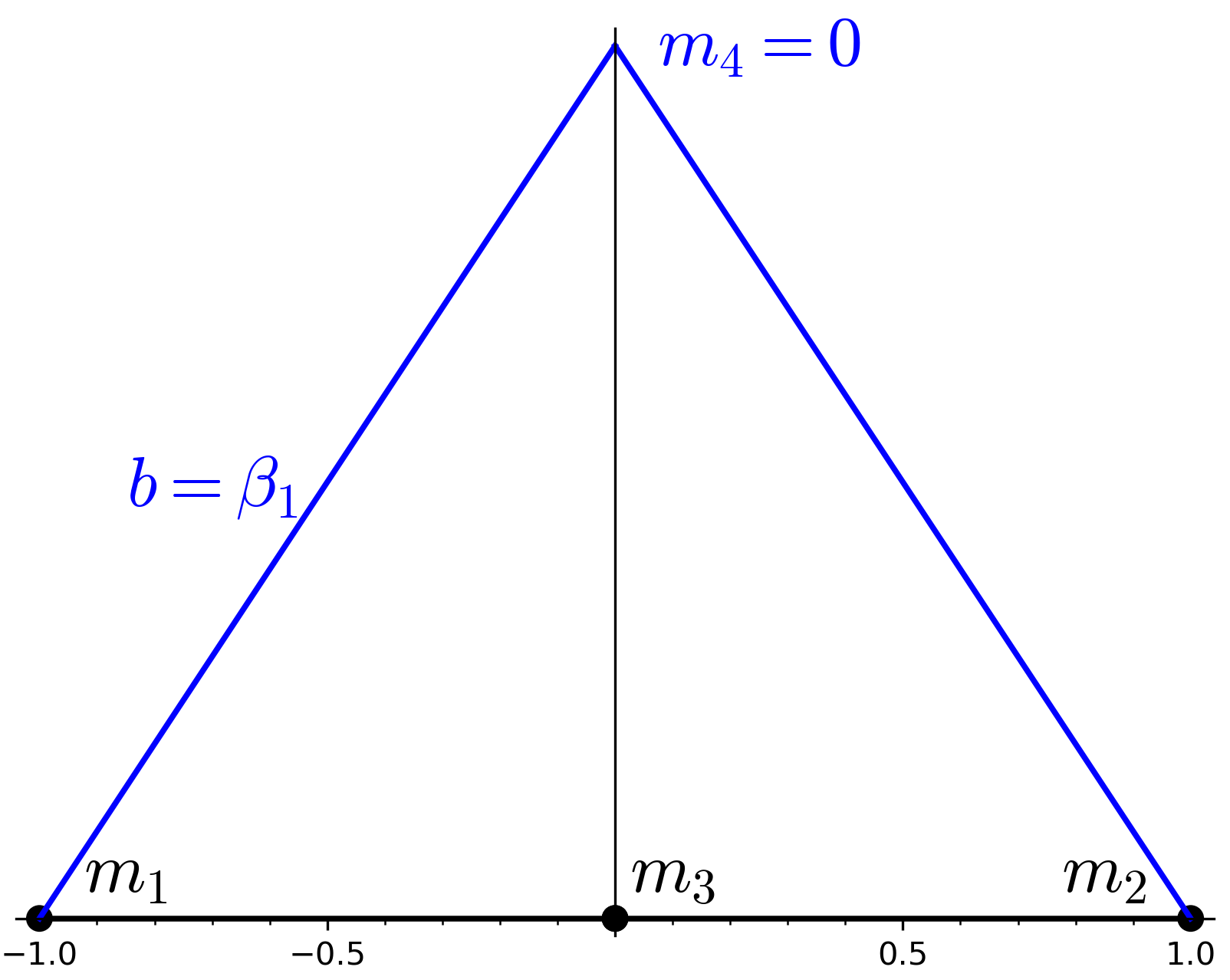}
        \caption{$a=1<b$ with $m_4=\mu=0$. $m_1,m_2$ and $m_3$ are collinear with $m_3$ on the middle of $m_1$ and $m_2$.}
        \label{a=1}
    \end{minipage}
      \hfill
    \begin{minipage}[t]{0.4\textwidth}
        \centering
        \includegraphics[width=\textwidth]{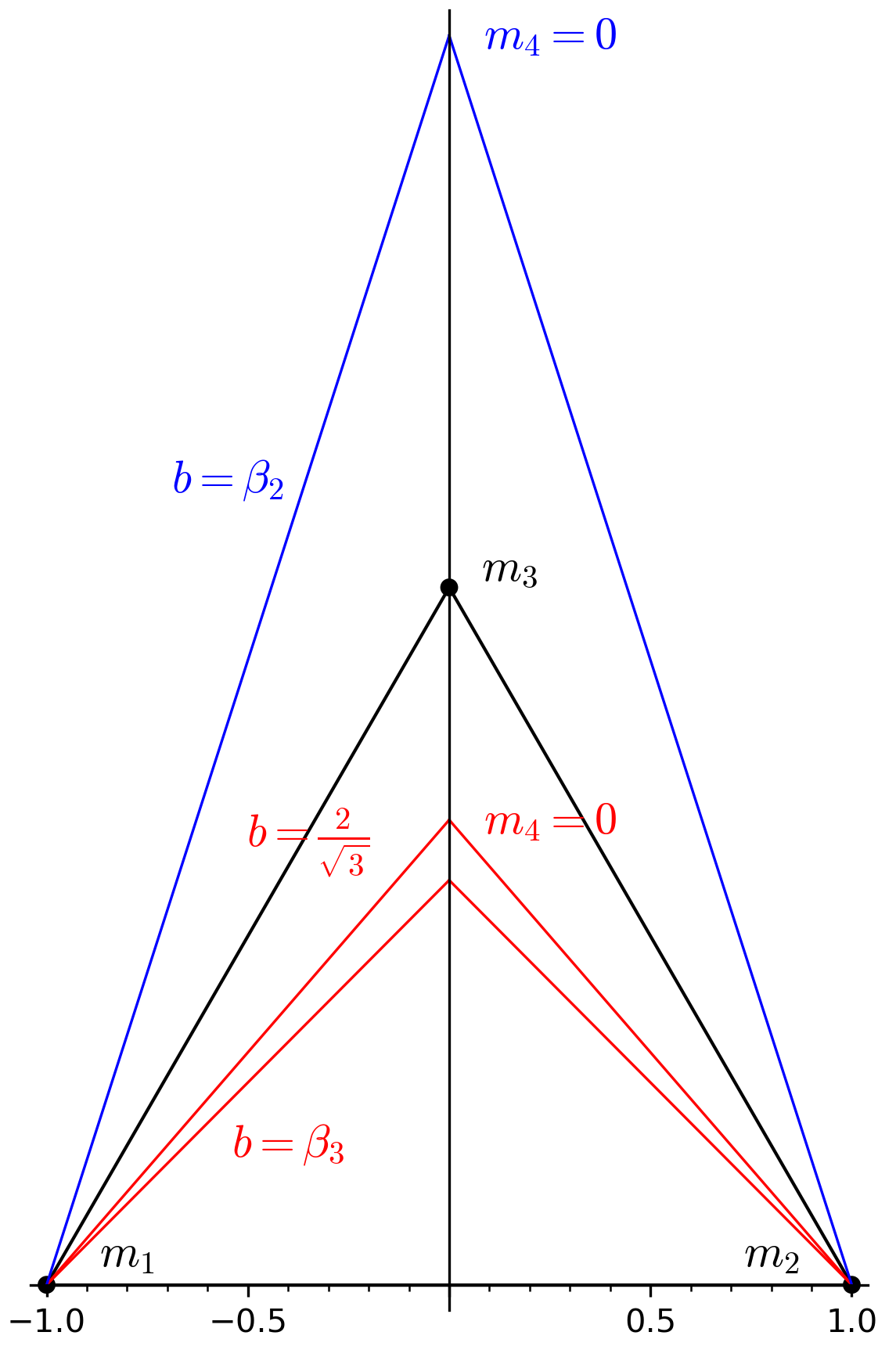}
        \caption{$a=2$ with $m_4=\mu=0$. The two equal legs in blue denote the outer case when $b>a$, and the red analogs denote the two limiting central configurations in the inner case with $b<a$.}
        \label{a=2}
    \end{minipage}
\end{figure}

In Section \ref{outersec} and Section \ref{innersec}, we will deal with the outer case and the inner case, respectively, since they correspond to two different equation systems.

\section{The outer case}\label{outersec}
For convenience, we denote by $I_1=(1,2)$ and $I_2=(\sqrt{2},\sqrt{2}+\sqrt{6})$ two open intervals.
\subsection{The proof of Theorem \ref{th1}: Part 1}
\begin{lemma}\label{partialgb}
	For $(a,b)\in D_1=(I_1\times I_2)\backslash \{(a,b)\vert a\geq b\}$, we have $\partial g_{outer}/\partial b>0$.
\end{lemma}
\begin{proof}
By direct computation, we have
\begin{equation*}
\frac{\partial g_{outer}(a,b)}{\partial b}
	=\frac{4b^5 \left(\frac{1}{a^3}-\frac{2}{\left(\sqrt{a^2-1}-\sqrt{b^2-1}\right)^3}\right)+b^5+16 b^2-24}{4 b^4 \sqrt{b^2-1}}>0
\end{equation*}
holds for $(a,b)\in (I_1\times I_2)\backslash \{(a,b)\vert a\geq b\}$, since $\frac{1}{a^3}-\frac{2}{\left(\sqrt{a^2-1}-\sqrt{b^2-1}\right)^3}>0$ holds for $a<b$, and $b^5+16 b^2-24>0$ holds for $b>\sqrt{2}$. 
\end{proof}
%This holds for $(a,b)\in D_{outer}$, since $D_{outer}\subset (I_1\times I_2)\backslash \{(a,b)\vert a\geq b\}$.
\begin{lemma}\label{outerlemma}
	For any $a\in I_1$, there exists a unique $b= b_{outer}(a)$ such that 
	\begin{enumerate}[label=\arabic*)]
		\item $g_{outer}(a,b)<0$ for $ \max\{a,\sqrt{2}\}<b<b_{outer}(a)$,
		\item $g_{outer}(a,b_{outer}(a))=0$, and 
		\item $g_{outer}(a,b)>0$ for $ b_{outer}(a)<b<\sqrt{2}+\sqrt{6}$.
	\end{enumerate}
	\end{lemma}
\begin{proof}
On the one hand, we have
\begin{equation*}
	g_{outer}(a,\sqrt{2}+\sqrt{6})=\frac{(2+\sqrt{3}-\sqrt{a^2-1})^3-a^3}{a^3(2+\sqrt{3}-\sqrt{a^2-1})^2}+\frac{3}{4}>0,
\end{equation*}
since $2+\sqrt{3}-\sqrt{a^2-1}-a=(2-a)+(\sqrt{3}-\sqrt{a^2-1})>0$ for $a\in I_1$. 
We have
\begin{equation*}
		g_{outer}(a,\sqrt{2})= -\frac{a^3+(\sqrt{a^2-1}-1)^3}{a^3(\sqrt{a^2-1}-1)^2}+\frac{1}{4}-\frac{1}{\sqrt{2}}<0
	\end{equation*}
since $a+\sqrt{a^2-1}-1=(a-1)+\sqrt{a^2-1}>0$ holds for $a\in I_1$. 
Then we consider the limit of 
\begin{equation*}
g_{outer}(a,b)=\frac{\check g_{outer}(a,b)}{4 a^3 b^3(\sqrt{a^2-1}-\sqrt{b^2-1})^2}
\end{equation*}
when $b\to a^+$, where $\check g_{outer}(a,b)=a^3 \left(b^3-8\right) \sqrt{b^2-1} (\sqrt{a^2-1}-\sqrt{b^2-1})^2-4b^3( a^3 +(\sqrt{a^2-1}-\sqrt{b^2-1})^3)$. 
 It is easy to see that the denominator of $g_{outer}(a,b)$ is positive for any $b\neq a$, and $\lim\limits_{b\to a^+}\check g_{outer}(a,b)=-4a^6<0$ holds for $a\in I_1$. Both together imply
 $$\lim\limits_{b\to a^+}g_{outer}<0.$$
 Hence, by the Intermediate Value Theorem, there exists at least one $b\in (\max\{a,\sqrt{2}\},\sqrt{2}+\sqrt{6})$ such that $g_{outer}(a,b)=0$. 

On the other hand, from Lemma \ref{partialgb}, for any $(a,b)\in D_{outer}\subset D_1$, we have $\partial g_{outer}/\partial b>0$. Hence, from the Implicit Function Theorem, $b= b_{outer}(a)$ is a differential function with respect to $a\in I_1$. One can see the black dashed curve in Figure \ref{outerfig}.
\end{proof}

From Lemma \ref{outerlemma} we know that $b=b_{outer}(a)$ is uniquely determined by $a\in I_1$. By substuting $b=b_{outer}(a)$ into the first expression in \eqref{outereqs} we obtain
\begin{equation*}
\mu_{outer}(a, b_{outer}(a))=\hat \mu_{outer}(a)=\frac{ b_{outer}^3(8-a^3)\sqrt{a^2-1}(\sqrt{a^2-1}-\sqrt{ b_{outer}^2-1})^2 }{4a^3( b_{outer}^3+(\sqrt{a^2-1}-\sqrt{ b_{outer}^2-1})^3)}>0,
\end{equation*}
which implies $\mu_{outer}=\hat \mu_{outer}(a)$ is a positive differential function with respect to the only variable $a\in I_1$. 

\begin{figure}[htbp]
        \begin{minipage}{0.5\textwidth}
        \centering
        \includegraphics[width=\linewidth]{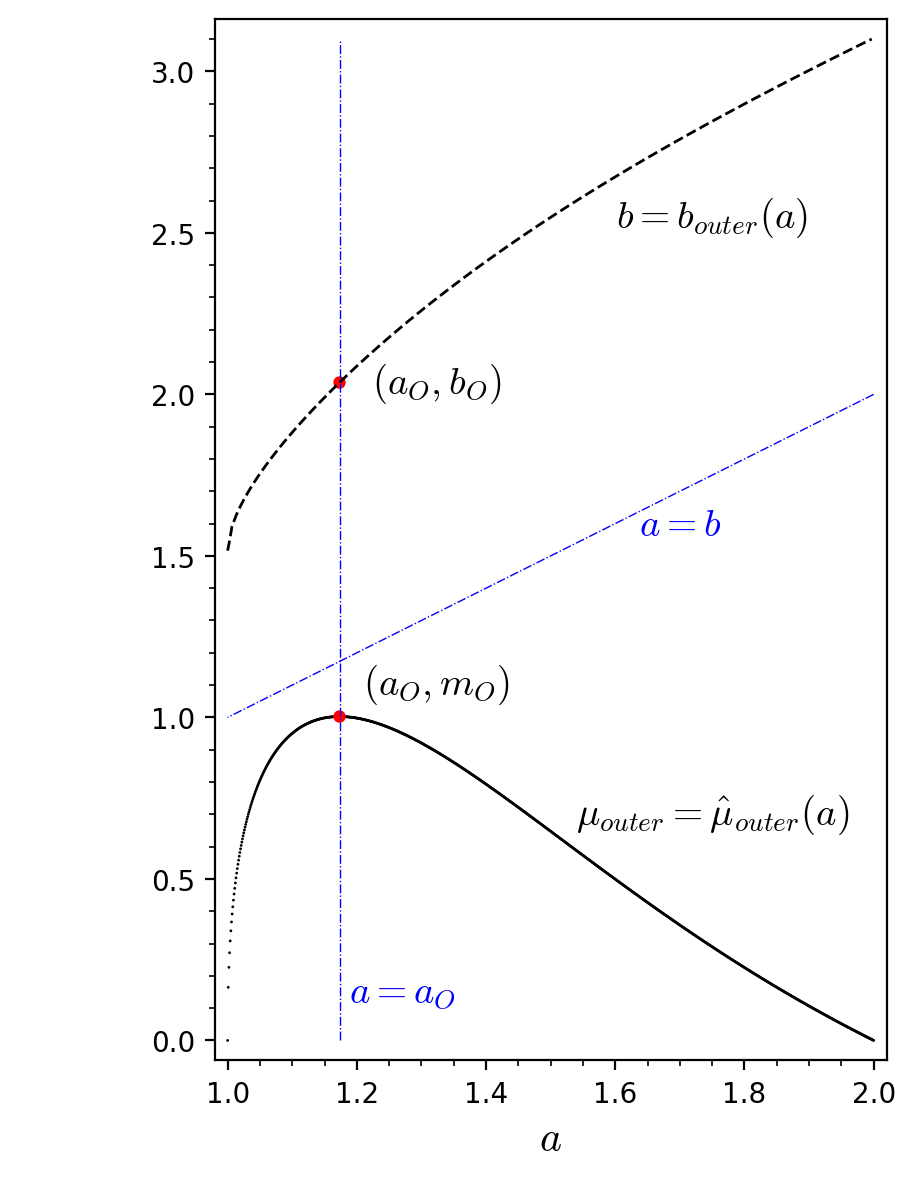}
    \end{minipage}
      \hfill
    \begin{minipage}{0.45\textwidth}
        \centering
        \caption{The function $b=b_{outer}(a)$ (the black dashed curve) and $\mu=\hat \mu_{outer}(a)$ (the black solid curve) with $a\in I_1$. The maximum point of $\hat \mu_{outer}(a)$ is around $(a,\mu)=(1.17318887031774929489004678$, $1.0026605475726100006858035)$, where $b\approx2.03659161617430193579842429$.}
       \label{outerfig}
    \end{minipage}
\end{figure}

\subsection{The proof of Theorem \ref{th2}: Part 1}\label{outermass}

In this section, we prove the first part of Theorem \ref{th2}. 
For the outer case, we notice that 
\begin{equation*}
	\frac{{\rm d}  \hat \mu_{outer}(a)}{{\rm d} a}=\frac{\partial \mu_{outer}}{\partial a}-\frac{\partial  \mu_{outer}}{\partial b}\cdot \frac{\frac{\partial g_{outer}}{\partial a}}{\frac{\partial g_{outer}}{\partial b}}.
\end{equation*}
By using the computer-assisted interval arithmetic introduced in Section \ref{IntK}, we show rigorously that the smooth function $\hat \mu_{outer}(a)$ possesses a unique maximum $\mu=m_{O}$ when $a\in I_1$. More precisely, we consider the equation $F_{outer}=(f_{outer},g_{outer})=0$, where
\begin{equation}\label{uniquedm}
\left\{
	\begin{aligned}
		0=&\frac{{\rm d} \hat \mu_{outer}(a)}{{\rm d} a}\cdot \frac{\partial g_{outer}}{\partial b}=f_{outer}(a,b),\\
		0=&g_{outer}(a,b).
	\end{aligned}
	\right.
\end{equation}
We aim to show $F_{outer}(a,b)=0$ has a unique solution in the interior of $D_1$. In fact, after excluding most of the interval pieces, we lock a small interval $ D_O=[1.17,1.75]\times [2.03,2.04]$ containing possible zeros. The domain $ D_O$ is first partitioned by dividing each dimensional interval into 200 equal parts. Each resulting subinterval is then evaluated with the Krawczyk operator from the interval arithmetic program, and all subintervals found to contain no zeros are excluded. Finally, we find the only solution $(a_{O},b_{O})$ of \eqref{uniquedm} in
\begin{equation}\label{uniquesol}
\left\{
	\begin{aligned}
 a_{O}=&1.17318887031774929489004678?\\
 b_{O}=&2.03659161617430193579842429?\\
	\end{aligned}
	\right.,
\end{equation}
and the corresponding mass is $\mu=m_{O}=1.00266054757261000068580350\textbf{216}?$.
%$$$$
In SageMath, the trailing question mark indicates it is an element (with a small uncertainty) in the real interval $[1.\cdots\textbf{21}, 1.\cdots\textbf{22}]$.
From Lemma \ref{partialgb} we know that $\partial g_{outer}(a,b)/\partial b>0$ for $(a,b)\in D_{outer}$, hence ${\rm d}  \hat \mu_{outer}(a)/{\rm d} a$ shares the same sign as $f_{outer}$. 
By substituting $a_1=1.1<a_{O}$ into $g_{outer}(a,b)=0$ we obtain $b_1=1.8823116918766211735913884695038?$. Then by the interval arithmetic, we have $f_{outer}(a_1,b_1)=5.338870446069789181020951478\textbf{74}?\subset [5.\cdots\textbf{7},5.\cdots\textbf{8}]$, i.e., $f_{outer}(a_1,b_1)>0$. This implies $\hat \mu_{outer}(a)$ increases in $a\in (1,a_{O})$ and decreases in $a\in (a_{O},2)$ with $a_{O}$ the unique maximum in $I_1$. One can see the solid black curve in Figure \ref{outerfig}. 

In addition, we consider the limitations at the two ends of $I_1$, and without much difficulty, we have
\begin{equation*}
	\begin{aligned}
		\lim\limits_{a\to 1^+}\hat \mu_{outer}(a,b)=&0,\\
		\lim\limits_{a\to 2^-}\hat \mu_{outer}(a,b)=&0.
	\end{aligned}
\end{equation*}
This implies that $\hat \mu_{outer}(a,b)$ can be continued to $a=1$ and $a=2$. Then we have the following results
\begin{enumerate}
	\item When $\tilde \mu\in (0,m_{O})$, the equation $\hat \mu_{outer}(a)=\tilde \mu$ has two solutions, namely, there are exactly two types of central configurations with the basic settings;
	\item Similarly, when $\tilde \mu =m_{O}$, there is a unique central configuration;
	\item When $\tilde \mu>m_{O}$, there is no such concave central configuration;
	\item In addition, when $\tilde \mu=0$ is admissible, there are three types of strictly concave central configurations and one type with three of the four bodies collinear. 
\end{enumerate}

%%%%%%%%%%%%%%%%%%%%
%%%%%%%%%%%%%%%%%%%%
%%%%%%%%%%%%%%%%%%%%
%%%%%%%%%%%%%%%%%%%%
%%%%%%%%%%%%%%%%%%%%
%%%%%%%%%%%%%%%%%%%%
%%%%%%%%%%%%%%%%%%%%
%%%%%%%%%%%%%%%%%%%%
%%%%%%%%%%%%%%%%%%%%

\subsection{Bifurcation for the outer case in the reduced subspace}\label{outerb}
We denote by $r=(a,b)$ the variables, and $\mu$ the parameter.
We consider the following system
\begin{equation*}
	H_{outer}(r,\mu)=(\mu-\mu_{outer},g_{outer})=0.
\end{equation*}
The Jacobian of $H_{outer}(r,\mu)$ with respect to $r$ is 
\begin{equation*}
J_{O}=
	\begin{bmatrix}
		j_{11}&j_{12}\\
		j_{21}&j_{22}
	\end{bmatrix}
	=\begin{bmatrix}
		-\frac{\partial \mu_{outer}}{\partial a}&-\frac{\partial \mu_{outer}}{\partial b}\\
		\frac{\partial g_{outer}}{\partial a}&\frac{\partial  g_{outer}}{\partial b}
	\end{bmatrix}.
\end{equation*}
It is easy to see that 
\begin{equation*}
	\det J_{O}=-f_{outer}(a,b).
\end{equation*}
From Section \ref{outermass} we know that $f_{outer}(a,b)=0$ has a unique zero $r_{O}=(a_{O},b_{O})$ in $D_{outer}$, and this implies $r_{O}$ is the only bifurcation point in the symmetric reduced space. In addition, we do some simple interval arithmetic in the following according to Theorem  \ref{typeofbifur} in Section \ref{bifurcations} to find that $r_{O}$ is a bifurcation point of a fold type. 
The eigenvectors of the eigenvalue $0$ corresponding to $J_{O}$ and $J_{O}^T$ can be chosen as 
\begin{equation*}
	v_{O}=[-j_{12},j_{11}]^T \text{ and } w_{O}=[-j_{21},j_{11}]^T.
\end{equation*}
$\partial H_{outer}/\partial \mu=[1,0]^T$. Substituting $(r_{O},m_{O})$ into $J_{O}$ we have 
\begin{equation*}
	J_{O}(r_{O},m_{O})=
	\begin{bmatrix}
		4.516621366283169955012080485?&-2.323508104950223708098024492?\\
		-5.4682991424714572047489378495?& 2.8130844601394040989994648844?
	\end{bmatrix}.
\end{equation*}
Then we have 
\begin{equation*}
\left\{
	\begin{aligned}
		w_{O}^T\cdot \frac{\partial H_{outer}}{\partial m}(r_{O},m_{O})=&-j_{21}(r_{O},m_{O})\\
		=&5.468299142471457204748937849\textbf{5}?\\
		\in &[5.\cdots\textbf{4},5.\cdots\textbf{6}]\not\ni0,\\
		w_{O}^TD^2H_{outer}(r_{O},m_{O})(v_{O},v_{O})=&w^T\cdot[D^2(\mu-\mu_{outer}(r))(v_{O},v_{O}),D^2g_{outer}(r)(v_{O},v_{O})](r_{O},m_{O})\\
		=&431.1813805037941508341333\textbf{27265}?\\
		\in&[431.\cdots\textbf{2610},431.\cdots\textbf{2843}]\not\ni0,
	\end{aligned}
\right. 
\end{equation*}
which implies $r_{O}$ is a fold bifurcation point.

\section{The inner case}\label{innersec}
The inner case is almost the same. We denote by $I_3=(\beta_3,\beta_4)$ two intervals. 

\subsection{The proof of Theorem \ref{th1}: Part 2}
\begin{lemma}\label{innerpartialg}

For any $(a,b)\in D_2=(I_2\times I_3)\backslash \{(a,b)\vert a\leq b\}$, we have $\partial g_{inner}/\partial a<0$.
\end{lemma}
\begin{proof}
From \eqref{ginnerabtoa} we know $\frac{\partial g_{inner}(a,b)}{\partial a}<0$ holds for $\{(a,b)\vert 1<b<a\}\supset D_2$. 
\end{proof}

\begin{lemma}\label{innerlemma}
For any $b\in I_3$, there exists a unique $a=a_{inner}(b)$ such that 
\begin{enumerate}[label=\arabic*)]
		\item $g_{inner}(a,b)>0$ for $ \max\{b,\sqrt{2}\}<a<a_{inner}(b)$,
		\item $g_{inner}(a_{inner}(b),b)=0$, and 
		\item $g_{inner}(a,b)<0$ for $ a_{inner}(b)<a<\sqrt{2}+\sqrt{6}$.
	\end{enumerate}
\end{lemma}
\begin{proof}
By direct computation, we have
\begin{equation*}
	g_{inner}(\sqrt{2},b)=\frac{1}{4} \left((\sqrt{2} +1)\sqrt{b^2-1}+\frac{4}{\left(\sqrt{b^2-1}-1\right)^2}-\frac{8 \sqrt{b^2-1}}{b^3}-\sqrt{2}\right)>0
\end{equation*}
holds for $b\in I_3\backslash\{\sqrt{2}\}$. When $a\to b^+$, we consider the limitation of 
\begin{equation*}
	g_{inner}(a,b)=\frac{\check g_{inner}(a,b)}{4 a^3 b^3 \left(\sqrt{a^2-1}-\sqrt{b^2-1}\right)^2},
\end{equation*}
where $\check g_{inner}(a,b)=4 a^3 b^3+\sqrt{a^2-1} (-2 a^3 b^5+2 a^3 b^3+16 a^3 b^2-16 a^3-4 a^2 b^3-12 b^5+16 b^3)+\sqrt{b^2-1} (a^5 b^3-8 a^5+a^3 b^5-2 a^3 b^3-8 a^3 b^2+16 a^3+12 a^2 b^3+4 b^5-16 b^3)$. Without much difficulty, we have $\lim\limits_{a\to b^+}\check g_{inner}(a,b)=4b^6>0$, which implies 
\begin{equation*}
	\lim\limits_{a\to b^+} g_{inner}(a,b)>0.
\end{equation*}
At the same time, we have 
\begin{equation*}
	g_{inner}(\sqrt{2}+\sqrt{6},b)
	=\frac{\sqrt{b^2-1}-(2+\sqrt{3})}{\left(\sqrt{2}+\sqrt{6}\right)^3}+\frac{1}{(2+\sqrt{3}-\sqrt{b^2-1})^2}+\frac{\left(b^3-8\right) \sqrt{b^2-1}}{4 b^3}<0
\end{equation*}
holds for $b\in I_3$ by the interval arithmetic method and the Krawczyk operator. Hence, there exists at least one $a\in (\max\{b,\sqrt{2}\},\sqrt{2}+\sqrt{6})$ such that $g_{inner}(a,b)=0$.
In addition, from Lemma \ref{innerpartialg}, for any $(a,b)\in D_2$ we have $\partial g_{inner}(a,b)/\partial a<0$. Then from the Implicit Function Theorem, $a= a_{inner}(b)$ is a differential function with respect to $b\in I_3$. One can see the red and the green dashed curves in Figure \ref{innercase}.
\end{proof}

\begin{figure}[htbp]
	\begin{minipage}{0.49\textwidth}
	\centering
		\includegraphics[width=\textwidth]{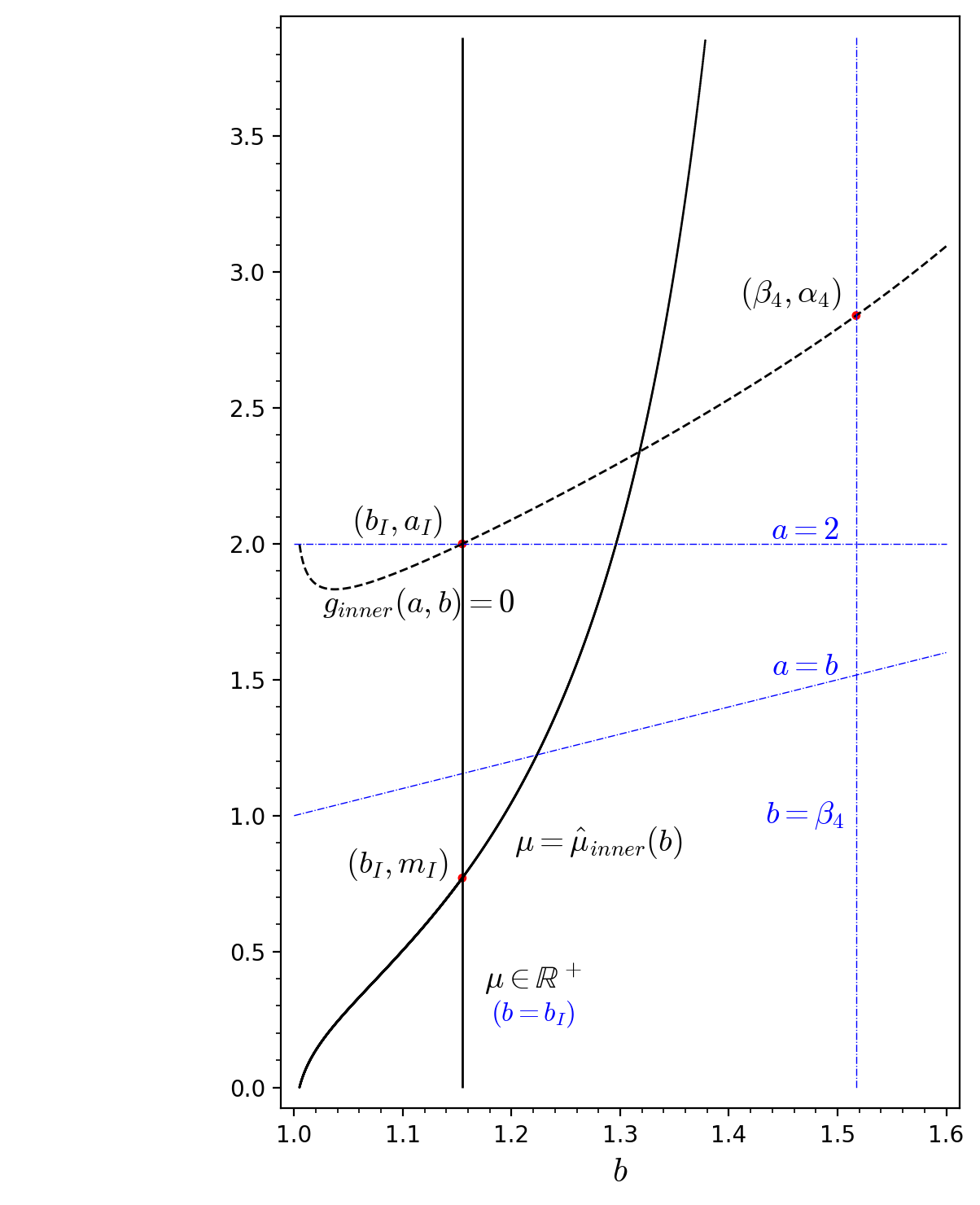}
	\end{minipage}
	\hfill
	\begin{minipage}{0.45\textwidth}
	\centering
		\caption{The function $a=a_{inner}(b)$ (the black dashed curve) and $\mu=\hat \mu_{inner}(b)$ (the black solid curve) with $b\in I_3$. The black vertical line denotes the ordinary solution, where $m_1,m_2,m_3$ form an equilateral triangle centered at $\mu\in (0,+\infty)$. These two curves intersect when $\mu=m_I$. The blue dashed curve denotes $h_{inner}(a,b)=0$, and it intersects $g_{inner}(a,b)=0$ at $(a,b)=(2,2/\sqrt{3})$ and $(a,b)=(\alpha_4,\beta_4)$.}
	\label{innercase}
	\end{minipage}
\end{figure}
From Lemma \ref{innerlemma} we know that $a=a_{inner}(b)$ is uniquely determined by $b\in I_3$. By substuting $a=a_{inner}(b)$ into the first expression in \eqref{innereqs} we obtain
\begin{equation*}
	\mu_{inner}( a_{inner}(b),b)=\hat \mu_{inner}(b)=\frac{b^3(8-a_{inner}^3)\sqrt{a_{inner}^2-1}(\sqrt{a_{inner}^2-1}-\sqrt{b^2-1})^2}{4 a_{inner}^3 ((\sqrt{a_{inner}^2-1}-\sqrt{b^2-1})^3-b^3)}>0,
\end{equation*}
and this implies $\mu_{inner}=\hat \mu_{inner}(b)$ is a positive differential function with respect to the only variable $b\in I_3$.

\subsection{The proof of Theorem \ref{th2}: Part 2}\label{prfth2-2}
We would like to prove that $\hat \mu_{inner}(b)$ is increasing, namely, for  $b\in I_3\backslash\{2/\sqrt{3}\}$ we have
\begin{equation*}
	\begin{aligned}
		\frac{{\rm d}  \hat \mu_{inner}(b)}{{\rm d} b}
		=&\frac{\partial \mu_{inner}}{\partial b}-\frac{\partial  \mu_{inner}}{\partial a}\cdot \frac{\frac{\partial g_{inner}}{\partial b}}{\frac{\partial g_{inner}}{\partial a}}=\frac{1}{\frac{\partial g_{inner}}{\partial a}}\cdot \begin{vmatrix}
			\frac{\partial \mu_{inner}}{\partial b}&\frac{\partial \mu_{inner}}{\partial a}\\
			\frac{\partial g_{inner}}{\partial b}&\frac{\partial g_{inner}}{\partial a}
		\end{vmatrix}>0.
	\end{aligned}
\end{equation*}
It is equivalent to proving 
\begin{equation}\label{dhatmdb}
\frac{{\rm d}  \hat \mu_{inner}(b)}{{\rm d} b}\cdot \frac{\partial g_{inner}}{\partial a}=
	\begin{vmatrix}
			\frac{\partial \mu_{inner}}{\partial b}&\frac{\partial \mu_{inner}}{\partial a}\\
			\frac{\partial g_{inner}}{\partial b}&\frac{\partial g_{inner}}{\partial a}
		\end{vmatrix}=\frac{\tilde f(a,b)}{\tilde {\tilde f}(a,b)}<0, 
\end{equation}
since from \eqref{ginnerabtoa} we have $\frac{\partial g_{inner}}{\partial a}<0$ when $1<b<a$.
To avoid possible singularities caused by the denominator $\tilde {\tilde f}(a,b)=16 a^7 \sqrt{a^2-1} b \sqrt{b^2-1} (\sqrt{a^2-1}-\sqrt{b^2-1})(4 (\sqrt{a^2-1}-\sqrt{b^2-1})+b^2 (-3 \sqrt{a^2-1}+\sqrt{b^2-1}+b)-a^2 (\sqrt{a^2-1}-3 \sqrt{b^2-1}))^2>0$, we want to show $\tilde f(a,b)<0$ when $b\in I_3/\backslash\{2/\sqrt{3}\}$. We denote by 
$$r_I=(a_I,b_I)=(2,2/\sqrt{3}).$$
Then we find that $\tilde f(r_I)=0$. We continue to show that $r_I$ is the unique solution for
\begin{equation}\label{dminner}
	\left\{
\begin{aligned}
0=&\tilde f(a,b),\\
0=&\tilde g(a,b),
\end{aligned}
\right.
\end{equation}
where $\tilde g(a,b)=g_{inner}(a,b)\cdot (4 a^3 b^3 \left(\sqrt{b^2-1}-\sqrt{a^2-1}\right)^2)$ is the numerator of $g_{inner}(a,b)$. 
We see 
$$\tilde f(a,b)=\tilde f(a_{inner}(b),b)$$ 
a differential function with respect to $b$ and consider its derivative. Similar to \eqref{dhatmdb}, we denote by 
\begin{equation*}
	\frac{{\rm d}\tilde f(a,b)}{{\rm d}b}\cdot \frac{\partial \tilde g(a,b)}{\partial a}=
	\begin{vmatrix}
		\frac{\partial \tilde f}{\partial b}&\frac{\partial \tilde f}{\partial a}\\
		\frac{\partial \tilde g}{\partial b}&\frac{\partial \tilde g}{\partial a}\\
	\end{vmatrix}=\frac{\tilde h(a,b)}{\sqrt{a^2-1}\sqrt{b^2-1}}.
\end{equation*}
We show that the following system possesses a unique solution $r=r_I$
\begin{equation}\label{tildeh=0}
	\left\{
\begin{aligned}
0=&\tilde h(a,b),\\
0=&\tilde g(a,b).
\end{aligned}
\right.
\end{equation}
It is easy to check that $r=r_I$ is a solution for \eqref{tildeh=0}. 
We use the interval arithmetic and the Krawczyk operator to show that it is unique.
After a somewhat tedious process of ruling out most intervals, we locate the solution in $(a,b)\in D_2=[1.999,2.001]\times [1.154,1.155]$. By dividing each 1-dim interval into 50 pieces, we find the only solution in a small 2-dim interval shown as 
\begin{equation*}
	\left\{
	\begin{aligned}
		a =&2.00000416563379439085804844?\\
		b=&1.15470274245007469162818378?
	\end{aligned}
	\right..
\end{equation*}
Since $r_I$ is already a solution in this interval, we say that it is the unique one. 
Substituting $b_3=1.1$ and $b_4=1.3$ into $\tilde g_{inner}(a,b)=0$ and solving the equation for $a$ respectively, we obtain two intervals for the solutions to exist
\begin{equation*}
	\begin{aligned}
		A_3=&[1.90274898693561,1.90274898693562],\\ % (+,-)
		A_4=&[2.29936985134914,2.29936985134915]. % (+,-)
	\end{aligned}
\end{equation*}
By direct computation and \eqref{ginnerabtoa} we have 
\begin{equation*}
	\begin{aligned}
		\tilde h(A_3,1.1)=&7.020915634079450852084143\times 10^7?>0,\\
		\tilde h(A_4,1.3)=&-1.7768962931900663749738920\times 10^{10}?<0. 
	\end{aligned}
\end{equation*}
This implies $b=2/\sqrt{3}$ is a maximum for $\tilde f(a_{inner}(b),b)$. Since $\tilde f(r_I)=0$, we have $\tilde f(a,b)<0$ holds for all $(a,b)\in D_{inner}\backslash\{(2,2/\sqrt{3})\}$. 
Finally, we conclude that \eqref{dhatmdb} holds for all $(a,b)\in D_{inner}\backslash\{(2,2/\sqrt{3})\}$. 
This implies that $\mu=\hat \mu_{inner}(b)$ is increasing with respect to $b$. One can see the solid black curve labeled $\mu_{inner}=\hat \mu_{inner}(b)$ in Figure \ref{innercase} for illustration. In addition, when $r=r_{I}$, we have $\mu_{inner}\in \mathbb{R}^+$ from the expression in \eqref{mab}, and one can see the solid black vertical line in Figure \ref{innercase}. Both curves intersect at $(b,\mu)=(b_I,m_I)$, where $m_I=\frac{64\sqrt{3}+81}{249}$. 

In fact, if we see $\mu$ as a parameter and transform \eqref{mab} into $j(\mu,a,b)=\mu\cdot y_1-x=0$, then by solving the equation
\begin{equation}\label{jmab}
	\frac{{\rm d}j(\mu,a,b)}{{\rm d}b}\Big|_{(a,b)=r_I}=0
\end{equation}
we obtain $\mu=m_I$.
Then we have 
\begin{enumerate}
	\item For any given $\mu\in \mathbb{R}^+\backslash\{m_I\}$, there are exactly two central configurations. One possesses an isosceles triangle shape, and the other an equilateral triangle shape.
	\item When $\mu=m_I=\frac{64\sqrt{3}+81}{249}$, there is a unique central configuration of an equilateral triangle shape, and it is also the intersection of the two mass curves.
	%\item If $m=0$ is admissible, there are two central configurations according to 3.3) in Lemma \ref{rangeab}.
\end{enumerate}

\subsection{Bifurcation for the inner case in the reduced subspace}\label{innerb}
We use some notations in Section \ref{prfth2-2}. We see $\mu$ as a parameter and consider the following equations to avoid the possible singularities
\begin{equation*}
	H_{inner}(r,\mu)=(j(\mu,a,b),\tilde g(a,b))=0.
\end{equation*}
The Jacobian of $H_{inner}$ with respect to $r$ is 
\begin{equation*}
J_{I}=
	\begin{bmatrix}
		p_{11}&p_{12}\\
		p_{21}&p_{22}
	\end{bmatrix}
	=\begin{bmatrix}
		\mu\frac{\partial y_1}{\partial a}-\frac{\partial x}{\partial a}&\mu\frac{\partial y_1}{\partial b}\\
		\frac{\partial \tilde g}{\partial a}&\frac{\tilde g}{\partial b}
	\end{bmatrix}.
\end{equation*}
From \eqref{jmab} we know that the solution of $\det J_{I}=0$ is the same as the solution of \eqref{dminner}, namely $r_I=(a_I,b_I)$ with $\mu=m_I$.
The eigenvectors of the eigenvalue $0$ corresponding to $J_{I}$ and $J_{I}^T$ can be chosen as 
\begin{equation*}
	v_{I}=[-p_{12},p_{11}]^T \text{ and } w_{I}=[-p_{21},p_{11}]^T.
\end{equation*}
$\partial H_{inner}/\partial \mu=[y_1,0]^T$. Substituting $(r_{I},m_{I})$ into $J_{I}$ we have 
\begin{equation*}
	J_{I}(r_{I},m_{I})=0.
\end{equation*}
Then we have 
\begin{equation*}
\left\{
	\begin{aligned}
		w_{I}^T\cdot \frac{\partial H_{inner}}{\partial \mu}(r_{I},m_{I})=&-p_{21}(r_{I},m_{I})\cdot y_1(r_{I},m_{I})=0,\\
		w_{I}^T\cdot D^2H_{inner}(r_{I},m_{I})(v_{I},v_{I})=&w^T\cdot[D^2( j(\mu,r))(v_{I},v_{I}),D^2\tilde g(r)(v_{I},v_{I})](r_{I},m_{I})\\
		=&-678.751441394201336948542495?\neq 0,
	\end{aligned}
\right. 
\end{equation*}
which implies $r_{I}$ is a transcritical bifurcation point.

\section{Numerical analysis for the bifurcations in the whole planar 4-body configuration space}\label{secbif}
In this section, we present global images of the bifurcations in both cases, shown in Figures \ref{outerbif} and \ref{innerbif}, based on prior studies and numerical results. 
To illustrate it clearly, we always set $m_1=m_2=m_3=1$, $m_4=\mu$ and $q_1=(-1,0),q_2=(1,0)$. 
The only two bifurcation points found in Section \ref{outerb} and \ref{innerb} are also bifurcations in the whole planar 4-body configuration space, which are denoted by $\mathcal{F}$ and $\mathcal{T}$ respectively.

%\subsection{The outer case}
In the outer case, except for the fold bifurcation point $\mathcal{F}$ already found in \cite{bernat2009} (denoted by $m^\ast$ in the literature), there exists another bifurcation point $\mathcal{P}$ of a supercritical pitchfork type with $\mu=m_{P}\approx 0.991842274390941$ found in \cite{rusu2016} (denoted by $m_{\ast\ast}$). 
Both bifurcation points can be obtained by perturbing from the concave equilateral triangle central configuration of the equal mass case $\mathcal{E}_1$ in two directions of $\mu$ according to \cite{rusu2016}.
With $\mu$ increasing from zero to $\infty$, as shown in Figure \ref{outerbif}, $\mathcal{P}$ appears first, and it gives rise to the asymmetric central configurations. The symmetric kite central configurations persist as $\mu$ increases, and coalesce into a single one at $\mu=m_O$. Then it ceases to exist. The whole process is illustrated in Figure \ref{outerbif}, with the detailed numerical values listed in Table \ref{outervalues}.

When $\mu\to +\infty$, it is equivalent to considering the relative equilibria of the $(1+3)$-body problem, in which the large mass is usually set to be $\mu=1$ and each trivial mass is set to be $m_i=\mu_i \cdot\varepsilon+O(\varepsilon^2)>0$ for $i=1,2,3$, with $\mu_i>0$ a fixed constant and $\varepsilon\to 0$. 
According to \cites{hall1988,moeckel1997}, these trivial masses will form clusters during the asymptotical approach, and they must converge to the circle centered at the only large mass with radius determined by this mass value. During the process, a collision may or may not occur. If there is no collision, namely, each trivial mass itself form a cluster, then from the study in \cites{hall1988, moeckel1997,corbera2015,casasayas1994} and up to the permutation of the vertices, there are three types of the relative equilibria of the $(1+3)$-body problem \cites{corbera2015,hall1988,casasayas1994} when $\mu_1=\mu_2=\mu_3$: one is convex, and the rest two belong to the two families in the concave inner case respectively, which are denoted by $\mathcal{W}_1$ and $\mathcal{W}_2$ in Figure \ref{innerbif} and mentioned in Remark \ref{alphabeta4}. 

While based on the numerical computation, we find that when $\mu\to+\infty$ (or equivalently $m_4/\mu=1$, and $m_1/\mu=m_2/\mu=m_3/\mu$ tend to 0), the limiting position for the asymmetric central configurations contains collision, which is not discussed in detail in the studies mentioned above. By our observation, the limiting positions of the four bodies form an equilateral triangle, with two of the zero masses colliding on the circle centered at the large mass. The process is shown in the asymmetric branch in Figure \ref{outerbif}.
%\deleted{If we keep $q_1=(-1,0)$, $q_2=(1,0)$ and set $m_1=1$, $m_2=m_3=m_4=\mu$, we see clearly that the positions of $m_3$ and $m_4$ tend to $(0,\sqrt{3})$ from the inside and from the outside of the circle centered at $q_1$, with radius $r_{12}=2$. }%The limit positions for the points form an equilateral triangle, as shown in Figure \ref{minfinite}. 

For the inner case, the two families of the concave kite central configurations, i.e., the equilateral and the isosceles triangle families, are shown in Figure \ref{innerbif}, with the numerical values of $q_3$ and $q_4$ of the isosceles family listed in Table \ref{innervalues}. 

% This is the basic assumption by Moeckel in \cite{moeckel1997} when considering the clusters of the infinitesimal masses.
%  It is well known that when $\mu_1=\mu_2=\mu_3$, 
%The limiting collision central configuration shown in Figure \ref{outerbif} and \ref{vars23_total} was mentioned in \cite{hall1988,moeckel1997} but not discussed in detail. We will elaborate further in Section \ref{secbif}.

%it turns to the study of the $(1+3)$-body problem, which has been discussed or mentioned in 
%

%Precisely in this limit case, two of the three trivial masses form clusters, and the limit shape of this $(1+3)$-body problem is an equilateral triangle with $m_3$ colliding with $m_1$ (or $m_2$); one can see $\mathcal{I}_3$ and $\mathcal{I}_4$ in Figure \ref{outerbif}. 
%If we fix the large mass $m_4$ at $(-1,0)$ and $m_1$ at $(1,0)$, and iterate from the equal mass case, it is clearer to see numerically that the positions of $m_2$ and $m_3$ tend to $(0,\sqrt{3})$ from different directions when $m=m_4$ tends to $\infty$. This is a point on the circle centered at $m_4$ with radius $r_{14}=2$, and the limit shape is an equilateral triangle; see Figure \ref{vars23_total}. 

\begin{figure}[htpb]
    \centering
    \includegraphics[width=\textwidth]{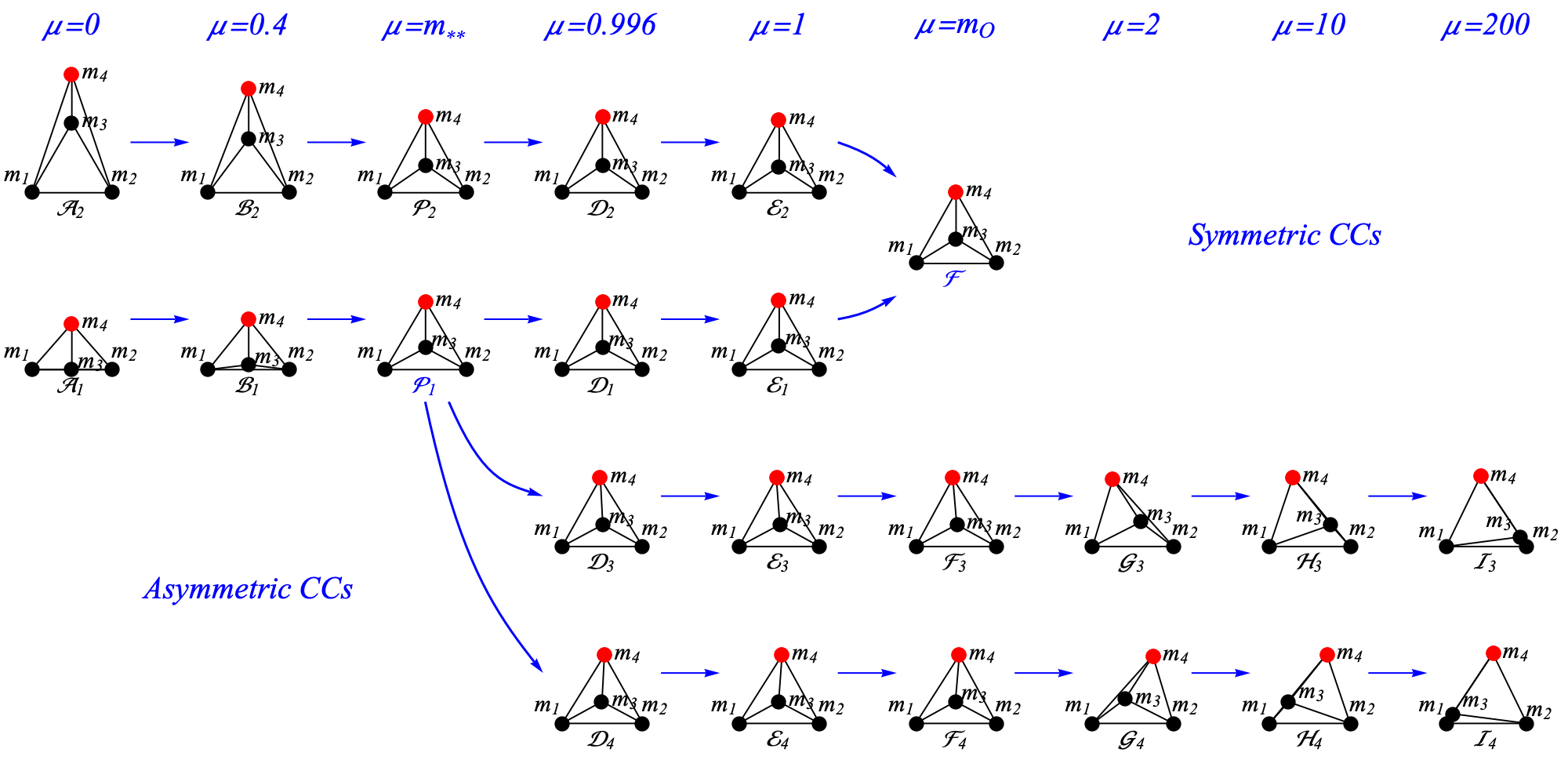}
    \caption{The global bifurcation picture for the outer case, including both symmetric and asymmetric central configurations. With $\mu$ increasing from zero to $\infty$, two bifurcation points appear successively. The symmetric kite central configurations coalesce into $\mathcal{F}$ and then cease to exist. The pitchfork point $\mathcal{P}$ gives rise to the asymmetric central configurations, and the limit position forms an equilateral triangle, with two of the small masses colliding with each other on the circle centered at the large mass. $\mathcal{E}_3$ corresponds to $\mathcal{E}_2$ with its masses permuted, hence its symmetry axis is not the $y$-axis.}
    \label{outerbif}
\end{figure}

%\begin{figure}[htbp]
%	\begin{minipage}{0.65\textwidth}
%	\centering
%		\includegraphics[width=\textwidth]{minfinite.png}
%	\end{minipage}
%	\hfill
%	\begin{minipage}{0.35\textwidth}
%	\centering
%		\caption{We keep $q_1$ and $q_2$ still fixed on the $x$-axis symmetrically, and set $m_1=1$, $m_2=m_3=m_4=\mu$. Then the numerical asymptotic approach for the two trivial masses $m_4$ (in blue) and $m_3$ (in black) to the collision point $(0,\sqrt{3})$ is clearly shown as $\mu\to0^+$. The dashed red curve denotes part of the circle with radius 2 centered at $m_1$. The number of iterations is around 1000.}
%    \label{minfinite}
%	\end{minipage}
%\end{figure}
%

%\subsection{The inner case}
\begin{figure}[htpb]
    \centering
    \includegraphics[width=\textwidth]{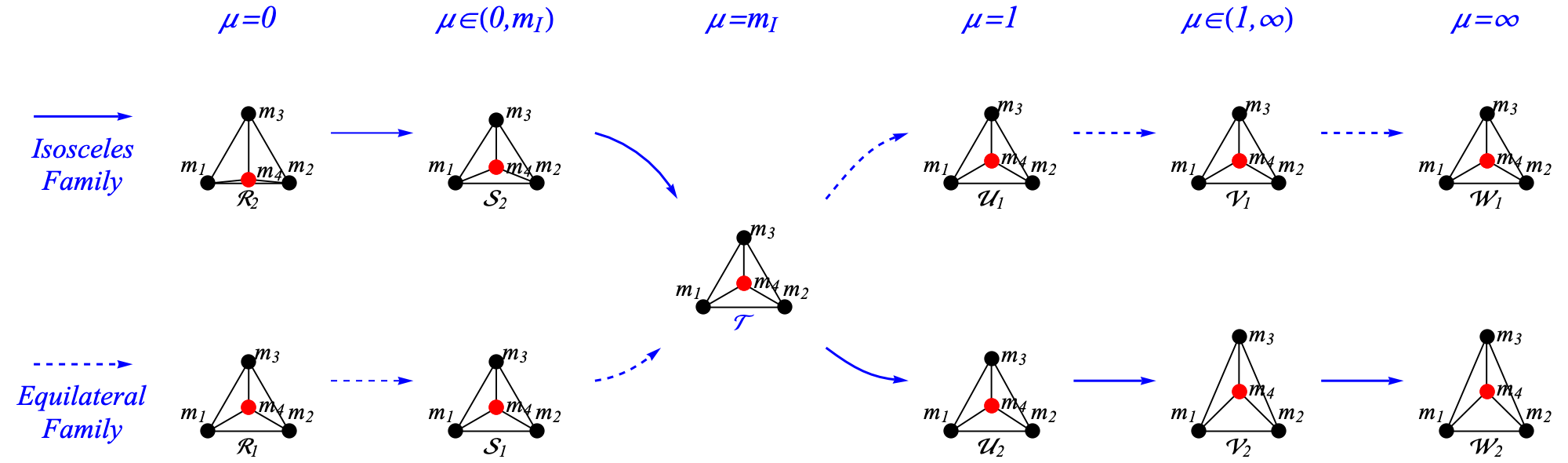}
    \caption{The global bifurcation for the inner case. There are two families, the isosceles triangle family and the equilateral triangle family, and they intersect at $\mathcal{T}$ when $\mu =m_I$, which is a transcritical bifurcation point.}
    \label{innerbif}
\end{figure}

 \counterwithin{table}{section}
\begin{table}[htbp]
\centering
\begin{tabular}{@{}cc|c c@{}}
\toprule
$\mu_{outer}$ & CCs & $q_3$ &  $q_4$  \\
\midrule
\multirow{2}{*}{0} 
& $\mathcal{A}_1$ & $(0,0)$ & $(0,1.1394282250)$ \\
 & $\mathcal{A}_2$ & $(0,\sqrt{3})$ & $(0,2.9373470790)$ \\
\cline{1-4}
\multirow{2}{*}{0.4} 
& $\mathcal{B}_1$ & $(0,0.1216979505)$ & $(0,1.2430125163)$ \\
 & $\mathcal{B}_2$ & $(0,1.3363457886)$ & $(0,2.5666495320)$ \\
\cline{1-4}
\multirow{2}{*}{$m_{P}$} 
&  \textcolor{blue}{$\mathcal{P}_1$} 
& $(0,0.5412989378)$ & $(0,1.6901724387)$ \\
 & $\mathcal{P}_2$ & $(0,0.6887261772)$ & $(0,1.8619952062)$  \\
\cline{1-4}
\multirow{4}{*}{0.996} 
& $\mathcal{D}_1$ & $(0,0.5566204325)$ & $(0,1.7079477108)$ \\
 & $\mathcal{D}_2$ & $(0,0.6722321579)$ & $(0,1.8427516451)$ \\
 & $\mathcal{D}_3$ & $(0.0101779325, 0.5418766314)$ & $(-0.0502922032, 1.6898418940)$ \\
 & $\mathcal{D}_4$ & $(-0.0101779325, 0.5418766314)$ & $(0.0502922032, 1.6898418940)$   \\
\cline{1-4}
\multirow{4}{*}{1} 
& $\mathcal{E}_1$ & $(0,\sqrt{3}/3)$ & $(0,\sqrt{3})$ \\
 & $\mathcal{E}_2$ & $(0,0.6503784730)$ & $(0,1.8172393947)$ \\
 & $\mathcal{E}_3$ & $(0.0142776898, 0.5424284291)$ & $(-0.0702774958, 1.6895283608)$ \\
 & $\mathcal{E}_4$ &  $(-0.0142776898, 0.5424284291)$ & $(0.0702774958, 1.6895283608)$   \\
\cline{1-4}
\multirow{3}{*}{$m_O$} 
& \textcolor{blue}{$\mathcal{F}$} & $(0,0.6134917485)$ & $(0,1.7741773900)$ \\
 & $\mathcal{F}_3$ & $(0.0164579892, 0.5427933036)$&   $(-0.0808016101, 1.6893222180)$\\
 & $\mathcal{F}_4$ &  $(-0.0164579892, 0.5427933036)$&   $(0.0808016101, 1.6893222180)$ \\
 \cline{1-4}
\multirow{2}{*}{2} 
&$\mathcal{G}_3$  & $(0.1921905901, 0.6061149070)$& $(-0.4950243426, 1.6717855273)$  \\
 & $\mathcal{G}_4$ &   $(-0.1921905901, 0.6061149070)$ & $(0.4950243426, 1.6717855273)$ \\
\cline{1-4}
\multirow{2}{*}{10}
 & $\mathcal{H}_3$ & $(0.5236764407, 0.5332602704)$ & $(-0.4191618074, 1.7159869404)$  \\
 & $\mathcal{H}_4$ &   $(-0.5236764407, 0.5332602704)$ & $(0.4191618074, 1.7159869404)$  \\
 \cline{1-4}
\multirow{2}{*}{$200$}
 & $\mathcal{I}_3$ & $(0.8398531867, 0.2391976896)$ & $(-0.1549114842, 1.7325613553)$ \\
 & $\mathcal{I}_4$ & $(-0.8398531867, 0.2391976896)$ & $(0.1549114842, 1.7325613553)$  \\
\bottomrule
\end{tabular}
\caption{Numerical values of $q_3$ and $q_4$ with different $\mu_{outer}$, including both symmetric and asymmetric cases. We set $m_1=m_2=m_3=1$ and $q_1=(-1,0),q_2=(1,0)$.}
\label{outervalues}
\end{table}

\begin{table}[htbp]
\centering
\begin{tabular}{@{}cc|c c@{}}
\toprule
$\mu_{inner}$ & CCs &  $q_3$& $q_4$ \\
\midrule
\multirow{1}{*}{0} 
&$\mathcal{R}_2$& $(0,\sqrt{3})$ &  $(0, 0.0994336508)$  \\
\cline{1-4}
\multirow{1}{*}{0.4} 
 &$\mathcal{S}_2$& $(0, 1.5757582838)$ &  $(0, 0.3968328465)$  \\
\cline{1-4}
\multirow{1}{*}{$m_I$} 
& \textcolor{blue}{$\mathcal{T}$} & $(0,\sqrt{3})$ & $(0,\sqrt{3}/3)$ \\
\cline{1-4}
\multirow{1}{*}{1} 
&$\mathcal{U}_2$& $(0, 1.8172393947)$ &  $(0, 0.6503784730)$  \\
\cline{1-4}
\multirow{1}{*}{5} 
&$\mathcal{V}_2$& $(0, 2.3369706923)$ &   $(0, 0.9858961071)$ \\
\cline{1-4}
\multirow{1}{*}{$+\infty$} 
&$\mathcal{W}_2$& $(0, 2.6580254426)$&  $(0, 1.1409031600)$  \\
\bottomrule
\end{tabular}
\caption{The numerical values of $q_3$ and $q_4$ for the isosceles triangle family in the inner cases derived from the interval method. We set $m_1=m_2=m_3=1$ and $q_1=(-1,0),q_2=(1,0)$.}
\label{innervalues}
\end{table}

\newpage
\section{Appendix: Basic theories}\label{sec2}
%\section{Basic theories}\label{sec2}
This appendix collects the necessary definitions and results on interval arithmetic and bifurcation theory that are used throughout the paper. The following material has been discussed in related work in previous work \cite{liu2026}, and is included here for completeness and ease of reference.
%In this section, we introduce some necessary concepts and conclusions that will be used later. The following aspects have already been mentioned in prior related work \cite{liu2026}; for completeness, we include them here. 
\subsection{The interval arithmetic method based on the Krawczyk operator%(\textcolor{red}{Say more})
}\label{IntK}
Suppose that $F(x)=(F_1(x),\cdots,F_N(x)):D\subset \mathbb{R}^N\to\mathbb{R}^N$ is a  $C^1$ function with $x\in\mathbb{R}^N$, and we use interval analysis to find zeros of the equation
\begin{equation}\label{equsys}
    F(x)=0.
\end{equation} 
The basic theory of the interval arithmetic as well as the Krawczyk operator can be found in \cites{krawczyk1969,alefeld2000,neumaier1990,moore2009}. One can also refer to \cites{moczurad2019,sun2023,rusu2016,liu2026} for specific applications in central configurations from different aspects. Here, we introduce only the necessary basic interval arithmetic. 

A real closed interval is denoted by $[u]=[u_1,u_2]$ with $u_1\leq u_2$. For the case $u_1=u_2$, it is just the single set $\{u\}$. We denote $u$ by an element in this interval as $u\in [u]$. 
The \textit{elementary operations} $\circ\in \Omega :=\{+,-,\cdot,/,\ast\ast\}$ are defined on the set of real intervals by putting 
\begin{equation*}
	[u]\circ [v ]:=  \{u\circ v\vert u\in [u],v\in [v]  \},
\end{equation*}
in which the performance of the division is restricted with $0\not\in v$. Similarly, the exponentiation $u\ast\ast v$ is restricted to one of the cases (i) $u_1>0$, (ii) $u_2\geq0$ and $v_2\geq0$, (iii) $0\not\in [u]$ and $[v]=\{v\}$ with $v\leq 0$ an integer, or (iv) $[v]=\{v\}$ with $v$ a positive integer.
The \textit{element functions} are the members of a predefined set $\Phi$ of real functions, continuous on every closed interval on which they are defined, such as $abs, sqr, sqrt, \exp,\ln, \sin,\cos, \arctan$.
%%%%%%

Suppose that $[x]=[x_{11},x_{12}]\times [x_{21},x_{22}]\times \cdots [x_{N1},x_{N2}]\subset \mathbb{R}^N$ is a closed interval, where $x_{i1},x_{i2}(i=1,\cdots, N)$ are two endpoints of a closed interval in each dimension. We denote by 
$$R(F;[x])= \{F(x)\vert x\in [x]\}=R(F_1;[x])\times \cdots \times R(F_N;[x])$$ 
the range of $F$ over $[x]$ and denote by 
$$F([x])=F_1([x])\times \cdots\times F_N([x]),$$ 
and the final interval under the interval computation. 
A simple but essential property of the interval arithmetic is \begin{equation}\label{intervalproperty}
   R(F;[x]) \subset F([x]).
\end{equation}
The Krawczyk operator is a modified version of the Newton operator based on the Brouwer fixed-point theorem.
We denote by $DF(x)$ the Jacobian matrix of $F$ at $x$. If $DF(x)$ is non-singular, we define the Krawczyk operator of $F$ as below:
\begin{equation}\label{kop}
    K(x_0,[x])=x_0-C\cdot F(x_0) +(Id-C\cdot DF([x])([x]-x_0)),
\end{equation}
with $C\in\mathbb{R}^{N\times N}$ a linear isomorphism. The properties of this operator are necessary
\begin{lemma}\label{koperator}
\begin{enumerate}
\item If $x^\ast \in [x]$ and $F(x^\ast)=0$, then $x^\ast\in K(x_0,[x])$. 
\item If $K(x_0,[x])\subset int[x]$, then there exists exactly one solution of equation $F(x)=0$ in $[x]$. This solution is non-degenerate with $DF(x)$ an isomorphism.
\item If $K(x_0,[x])\cap [x]=\varnothing$, then for all $x\in [x]$ we have $F(x)\neq0$. 
\end{enumerate}
\end{lemma}

%We design a program to find zeros of \eqref{equsys}.  
%The main idea of the program is to partition the $N$-dimensional interval into sufficiently small pieces and find zeros on these pieces one by one. 
%The property \eqref{intervalproperty} and the definition \eqref{kop}, together with Lemma \ref{koperator}, guarantee that this computer-assisted interval approach is rigid in mathematics. 

\subsection{Bifurcations}\label{bifurcations}
Suppose that
\begin{equation}\label{bifurF}
\begin{aligned}
    F:\mathbb{R}^N\times \mathbb{R}&\to \mathbb{R}^N\\
     (x,\mu)&\mapsto F(x,\mu),x\in \mathbb{R}^N,\mu\in\mathbb{R}
\end{aligned}
\end{equation}
is a smooth map and $\mu$ is a parameter. We follow the notation in Subsection \ref{IntK}, that $DF(x)$ is the Jacobian of $F$ with respect to $x$. We denote by $F_\mu$ the vector of partial derivatives of the components of $F$ with respect to $\mu$. 
When $\mu$ varies, the bifurcations of the system $F(x,\mu)=0$ may occur. The following theorem gives sufficient conditions for certain bifurcations. One can refer to \cites{kuznetsov2004,rusu2016} for more information. 
\begin{theorem}\label{typeofbifur}
    Suppose that $F(x_0,\mu_0)=0$, and the Jacobian matrix $J_F=DF(x_0,\mu_0)$ has a simple eigenvalue $\lambda_{O}=0$ with eigenvector $v$. Denote by $w$ the adjoint eigenvector, i.e., $J_F^T w=0$, where $J_F^T$ is the transpose of $J_F$. 
\begin{enumerate}
    \item If $\left\{\begin{aligned}
            w^TF_\mu(x_0,\mu_0)&\neq0\\
            w^T(D^2F(x_0,\mu_0)(v,v))&\neq0\end{aligned}\right.,$ \\then the system experiences a \textbf{fold bifurcation} at the equilibrium point $x_0$ as the parameter $\mu$ passes through the bifurcation value $\mu=\mu_0$.
    \item If $\left\{\begin{aligned}
            w^TF_\mu(x_0,\mu_0)&=0\\
            w^T(DF_\mu(x_0,\mu_0)v)&\neq0\\
            w^T(D^2F(x_0,\mu_0)(v,v))&\neq0\end{aligned}\right.,$ \\then the system experiences a \textbf{transcritical bifurcation} at the equilibrium point $x_0$ as the parameter $\mu$ passes through the bifurcation value $\mu=\mu_0$.
    \item If $\left\{\begin{aligned}
            w^TF_\mu(x_0,\mu_0)&=0\\
            w^T(DF_\mu(x_0,\mu_0)v)&\neq0\\
            w^T(D^2F(x_0,\mu_0)(v,v))&=0\\
            w^T(D^3F(x_0,\mu_0)(v,v,v))&\neq0\end{aligned}\right.,$ \\then the system experiences a \textbf{pitchfork bifurcation} at the equilibrium point $x_0$ as the parameter $\mu$ passes through the bifurcation value $\mu=\mu_0$. \\If $w^T(D^3F(x_0,\mu_0)(v,v,v))<0$, the branches occur for $\mu>\mu_0$, and the bifurcation is \textbf{supercritical}. Otherwise, the branches occur for $\mu<\mu_0$, and the bifurcation is \textbf{subcritical}.
\end{enumerate}
\end{theorem}

%\newpage 

%\bibliographystyle{amsplain}
%\bibliography{4bp3+1.bib}
\end{document}